\documentclass{amsart}

\usepackage{amsthm}
\usepackage{amsmath}
\usepackage{amssymb}
\usepackage{enumerate}
\usepackage{graphicx}
\usepackage[hidelinks,pdftex]{hyperref}
\usepackage{booktabs}
\usepackage{color}
\usepackage[dvipsnames]{xcolor}
\usepackage{import}
\usepackage{tikz-cd}
\usepackage{microtype}

\usepackage[margin=3cm]{geometry}

\AtBeginDocument{%
   \def\MR#1{}
}

\numberwithin{equation}{section}
\numberwithin{figure}{section}

\theoremstyle{plain}
\newtheorem{theorem}[equation]{Theorem}
\newtheorem{corollary}[equation]{Corollary}
\newtheorem{lemma}[equation]{Lemma}

\newtheorem{proposition}[equation]{Proposition}

\newtheorem*{namedtheorem}{\theoremname}
\newcommand{\theoremname}{testing}
\newenvironment{named}[1]{\renewcommand{\theoremname}{#1}\begin{namedtheorem}}{\end{namedtheorem}}

\theoremstyle{definition}
\newtheorem{definition}[equation]{Definition}
\newtheorem{remark}[equation]{Remark}
\newtheorem{construction}[equation]{Construction}

\newcommand{\from}{\colon} 

\newcommand{\HH}{{\mathbb{H}}}
\newcommand{\RR}{{\mathbb{R}}}
\newcommand{\ZZ}{{\mathbb{Z}}}

\newcommand{\CC}{{\mathbb{C}}}
\newcommand{\QQ}{{\mathbb{Q}}}
\newcommand{\calA}{{\mathcal{A}}}

\newcommand{\calV}{{\mathcal{V}}}

\newcommand{\refthm}[1]{Theorem~\ref{Thm:#1}}
\newcommand{\reflem}[1]{Lemma~\ref{Lem:#1}}
\newcommand{\refprop}[1]{Proposition~\ref{Prop:#1}}
\newcommand{\refcor}[1]{Corollary~\ref{Cor:#1}}

\newcommand{\refeqn}[1]{\eqref{Eqn:#1}}

\newcommand{\refdef}[1]{Definition~\ref{Def:#1}}
\newcommand{\refsec}[1]{Section~\ref{Sec:#1}}
\newcommand{\reffig}[1]{Figure~\ref{Fig:#1}}

\newcommand{\bdy}{\partial}

\newcommand{\len}{\operatorname{len}}

\title[Geometric triangulations and highly twisted links]{Geometric triangulations and highly twisted links}

\author{Sophie L. Ham}
\address{School of Mathematics,
Monash University,
VIC 3800, Australia}
\email{sophie.ham@monash.edu}

\author{Jessica S. Purcell}
\address{School of Mathematics,
Monash University,
VIC 3800, Australia}
\email{jessica.purcell@monash.edu}

\begin{document}

\begin{abstract}
  It is conjectured that every cusped hyperbolic 3-manifold admits a geometric triangulation, i.e.\ it is decomposed into positive volume ideal hyperbolic tetrahedra.
  Here, we show that sufficiently highly twisted knots admit a geometric triangulation.
  In addition, by extending work of Gu{\'e}ritaud and Schleimer, we also give quantified versions of this result for infinite families of examples. 
\end{abstract}

\maketitle

\section{Introduction}\label{Sec:Intro}

A \emph{topological triangulation} of a $3$-manifold $M$ is a decomposition of $M$ into tetrahedra or ideal tetrahedra such that the result of gluing yields a manifold homeomorphic to $M$. Every compact 3-manifold with boundary consisting of tori has interior that admits a topological ideal triangulation~\cite{Moise, Bing}. 

A \emph{geometric triangulation} is a much stronger notion. It is an ideal triangulation of a cusped hyperbolic 3-manifold $M$ such that each tetrahedron is positively oriented and has a hyperbolic structure of strictly positive volume, and such that the result of gluing gives $M$ a smooth manifold structure with its complete hyperbolic metric. It is still unknown whether every finite volume hyperbolic 3-manifold admits a geometric triangulation, and there are currently only a few families which provably admit one. These include 2-bridge knots and punctured torus bundles, due to Gu{\'e}ritaud and Futer~\cite{GueritaudFuter}, and all the manifolds of the SnapPy census~\cite{SnapPy}, as well as manifolds built from isometric Platonic solids~\cite{GoernerI,GoernerII}. On the other hand, Choi has shown that there exists an orbifold with no geometric triangulation~\cite{Choi}.

In this paper, we prove that a large family of knots admit geometric triangulations. To state the main result, we recall the following definitions. 

A \emph{twist region} of a link diagram consists of a portion of the diagram  where two strands twist around each other maximally. More carefully,
let $D(K)$ be a diagram of a link $K\subset S^3$. Two distinct crossings of the diagram are \emph{twist equivalent} if there exists a simple closed curve on the diagram that runs transversely through the two crossings, and is disjoint from the diagram elsewhere.
The collection of all twist equivalent crossings forms a \emph{twist region}.

Note that one can perform flypes on a link diagram until all twist equivalent crossings line up in a row, forming bigons between them. Suppose every simple closed curve that meets the diagram transversely only in two crossings has the property that it bounds a region of the diagram consisting only of bigons, or possibly contains no crossings. If this holds for every simple closed curve, the diagram is called \emph{twist-reduced}.
Figure~\ref{Fig:AugLinkConstruction}, left, shows an example of a twist-reduced diagram with exactly two twist regions.

The first main result of this paper is the following.

\begin{named}{\refthm{HighlyTwisted}}
For every $n\geq 2$, there exists a constant $A_n$ depending on $n$, such that if $K$ is a link in $S^3$ with a prime, twist-reduced diagram with $n$ twist regions, and at least $A_n$ crossings in each twist region, then $S^3-K$ admits a geometric triangulation.
\end{named}

The proof uses links called \emph{fully augmented links}. These are obtained by starting with a twist-reduced diagram of any knot or link, and for each twist-region, adding a simple unknot called a \emph{crossing circle} encircling the twist-region. We further remove all pairs of crossings in each twist region; see \reffig{AugLinkConstruction}. 
\begin{figure}
  \includegraphics{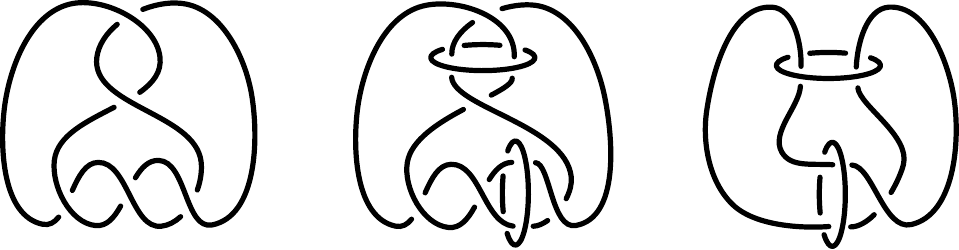}
  \caption{Constructing a fully augmented link.}
  \label{Fig:AugLinkConstruction}
\end{figure}
The result has explicit geometric properties, and can be subdivided into geometric tetrahedra.
The original link complement is obtained by Dehn filling the crossing circles. We complete the proof of \refthm{HighlyTwisted} by arguing that Dehn filling can be performed in a way that gives a geometric triangulation.

In fact, we can prove a result that is more general than \refthm{HighlyTwisted}, allowing any Dehn fillings on crossing circles and indeed leaving some crossing circles unfilled. This is the following theorem.

\begin{named}{\refthm{DehnFillingGeneral}}
  Let $L$ be a hyperbolic fully augmented link with $n\geq 2$ crossing circles. Then there exist constants $A_1, \dots, A_n$ such that if $M$ is a manifold obtained by Dehn filling the crossing circle cusps of $S^3-L$ along slopes $s_1, \dots, s_n$ whose lengths satisfy $\len(s_i)\geq A_i$ for each $i=1, \dots, n$, then $M$ admits a geometric triangulation. Allowing some collection of $s_i=\infty$, i.e.\ leaving some crossing circle cusps unfilled, also admits a geometric triangulation. 
\end{named}

Gu{\'e}ritaud and Schleimer considered geometric triangulations and Dehn filling~\cite{GueritaudSchleimer}. They showed that if a cusped manifold satisfies certain `genericity' conditions, then Dehn filling can be performed via geometric triangulation.\footnote{In fact, they showed something stronger: that Dehn filling gives a triangulation that is actually \emph{canonical}, i.e.\ dual to the Ford-Voronoi domain, but we will not consider canonical decompositions here.} Unfortunately, the usual geometric decomposition of a fully augmented link, as in \cite[Appendix]{Lackenby:AltVolume} or \cite{FuterPurcell, Purcell:FullyAugmented}, fails Gu{\'e}ritaud--Schleimer's genericity conditions. Nevertheless, we may adjust the decomposition to give a triangulation satisfying the Gu{\'e}ritaud--Schleimer conditions. This is the idea of the proof of \refthm{HighlyTwisted}.

Highly twisted knots, as in \refthm{HighlyTwisted}, are known to have other useful geometric properties. For example, they can be shown to be hyperbolic when there are at least six crossings in each twist region~\cite{FuterPurcell}. When there are at least seven, there are bounds on the volumes of such knots and links~\cite{FKP:DFVJP}. With at least 116 crossings per twist region, there are bounds on their cusp geometry~\cite{Purcell:Cusps}. The results of \refthm{HighlyTwisted} are not as nice as these other results, because we do not have an effective universal bound on the number of crossings per twist region required to guarantee that a knot admits a geometric triangulation. Nevertheless, we conjecture such a bound holds.

To obtain effective results, we need to generalise and sharpen results of Gu{\'e}ritaud and Schleimer, and we do this in the second half of the paper. This allows us to present two effective results, which guarantee geometric triangulations of new infinite families of cusped hyperbolic 3-manifolds. The first result is the following.

\begin{named}{\refthm{MainBorromean}}
Let $L$ be a fully augmented link with exactly two crossing circles. Let $M$ be a manifold obtained by Dehn filling the crossing circles of $S^3-L$ along slopes $m_1, m_2 \in (\QQ\cup \{1/0\}) - \{0, 1/0, \pm 1, \pm 2\}$.
Then $M$ admits a geometric triangulation.   
\end{named}

There are three fully agumented links with exactly two crossing circles; one is the Borromean rings and the others are closely related; these are shown in \reffig{BorromeanRings}. The Dehn fillings of these links include double twist knots, which were already known to admit geometric triangulations by~\cite{GueritaudFuter}. They also include large families of cusped hyperbolic manifolds that do not embed in $S^3$. More generally, we also show the following.

\begin{named}{\refthm{Main2Bridge}}
Let $L$ be a result of taking the standard diagram of a 2-bridge link, and then fully augmenting the link, such that $L$ has $n>2$ crossing circles (and no half-twists). Let $s_1, s_2, \dots, s_n \in \QQ\cup\{1/0\}$ be slopes, one for each crossing circle, that are all positive or all negative. 
Suppose finally that $s_1$ and $s_n$ are the slopes on the crossing circles on either end of the diagram, and the slopes satisfy: 
\[ s_1, s_n \notin \{0/1, 1/0, \pm 1/1, \pm 2/1\}, \mbox{ and }
s_2, \dots, s_{n-1} \notin \{0/1, 1/0, \pm 1/1\}.\]
Then the manifold obtained by Dehn filling $S^3-L$ along these slopes on its crossing circles admits a geometric triangulation.
\end{named}

For ease of notation, we will refer to a link such as $L$ in the above theorem as a \emph{fully augmented 2-bridge link}. That is, a fully augmented 2-bridge link is obtained by fully augmenting the standard diagram of a 2-bridge link. 

The Borromean rings and other links of \refthm{MainBorromean} form examples of fully-augmented 2-bridge links, and therefore instances of \refthm{MainBorromean} also follow from \refthm{Main2Bridge}. However, we prove the theorems separately to build up tools.

The manifolds included in \refthm{Main2Bridge} include many 2-bridge links, obtained by setting each $s_j = 1/m_j$ where $m_j$ is an integer with appropriate sign. Such a Dehn filling gives a 2-bridge link with a diagram with at least two crossings per twist region, an even number of crossings in each twist region, and conditions on signs of twisting. All 2-bridge links were already known to admit geometric triangulations~\cite{GueritaudFuter}. However, again \refthm{Main2Bridge} also includes infinitely many additional manifolds obtained by different Dehn fillings.

\subsection{More on geometric triangulations}

It is known that every cusped hyperbolic $3$-manifold has a decomposition into convex ideal polyhedra, due to work of Epstein and Penner~\cite{EpsteinPenner}. The convex polyhedra may be further subdivided into tetrahedra, but the result may not give a geometric triangulation. 
The difficulty is that the subdivision involves triangulating the polygonal faces of the polyhedra, and these triangulations may not be consistent with each other under gluing. To solve this problem, flat tetrahedra are inserted between identified faces of the polyhedra; see Petronio and Porti for more discussion~\cite{PetronioPorti}. 

If we pass to finite covers, then geometric triangulations exist by work of Luo, Schleimer, and Tillmann~\cite{LuoSchleimerTillmann}: Every cusped hyperbolic 3-manifold admits a finite cover with a geometric triangulation. 

If we relax the restriction that the tetrahedra glue to give a complete hyperbolic metric, and only require that the dihedral angles of each tetrahedra are strictly positive, and sum to $2\pi$ around each edge of the triangulation, then the result is called an \emph{angle structure} (or sometimes a \emph{strict angle structure}). Geometric triangulations admit angle structures. Moreover, Hodgson, Rubinstein, and Segerman show that many 3-manifolds admit an angle structure, including all hyperbolic link complements in $S^3$~\cite{HRS:StrictAngleStructs}. However, they note that the triangulations they find are not generally geometric.

There was some hope in the past that a class of triangulations introduced by Agol~\cite{Agol:Veering}, called \emph{veering triangulations}, give geometric triangulations. Indeed it was shown by Hodgson, Rubinstein, Segerman and Tillmann~\cite{HRST:Veering} and by Futer and Gu{\'e}ritaud~\cite{FuterGueritaud:Veering} that veering triangulations admit  angle structures. However Hodgson, Issa, and Segerman found a 13-tetrahedron veering triangulation that is not geometric~\cite{HodgsonIssaSegerman}, and recently Futer, Taylor, and Worden showed that a random veering triangulation is not geometric~\cite{FuterTaylorWorden}. Thus tools to exhibit geometric triangulations must come from other directions.

Why geometric triangulations? 
Various results become easier with geometric triangulations. For example, Neumann and Zagier showed that certain useful bounds exist on the volume of a hyperbolic 3-manifold that admits a geometric triangulation~\cite{NeumannZagier}, although this can be proven in general with more work; see Petronio--Porti~\cite{PetronioPorti}. Similarly, Benedetti and Petronio give a straightforward proof of Thurston's hypebolic Dehn surgery theorem using geometric triangulations~\cite{BenedettiPetronio}. Choi finds nice conditions on the deformation variety for manifolds admitting geometric triangulations~\cite{Choi}. In summary, such triangulations seem to lead to simpler proofs, and more manageable geometry.

\subsection{Organisation}
The paper is organised as follows. Sections~\ref{Sec:FullyAugLinks} through~\ref{Sec:DehnFilling} give the proof of \refthm{HighlyTwisted}, on more general highly twisted knots. We recall fully augmented links in \refsec{FullyAugLinks}, layered solid tori in Sections~\ref{Sec:LST} and~\ref{Sec:AngleStructures}, and put this together with Dehn filling in \refsec{DehnFilling}.

Sections~\ref{Sec:Borromean} through~\ref{Sec:DFBorromean} give the proof of \refthm{MainBorromean}, on Dehn fillings of links with two crossing circles, and build up the new machinery required for both \refthm{MainBorromean} and \refthm{Main2Bridge}.

Finally, in \refsec{2Bridge}, we complete the proof of \refthm{Main2Bridge}, on Dehn fillings of fully augmented 2-bridge links.

\subsection{Acknowledgements}
We thank B.~Nimershiem for helping us to improve the exposition in \refsec{AngleStructures}. We also thank the referees for their comments, which helped us improve the paper. 
Both authors were supported in part by the Australian Research Council.

\section{Fully augmented links}\label{Sec:FullyAugLinks}

The links of the main theorem, \refthm{HighlyTwisted}, are obtained by Dehn filling a parent link, called a fully augmented link. In this section, we review fully augmented links and their geometry, and show that they admit geometric triangulations.

Begin with any twist-reduced diagram of a link. As in the middle of \reffig{AugLinkConstruction}, for each set of twist equivalent crossings, insert a single unknotted curve that encircles the bigons of the twist region. If a twist region consists of only a single crossing, there are two ways to insert this link component; either will do. These unknotted components are chosen to be disjoint, and to bound discs that are punctured by exactly two strands of the original link. We call them \emph{crossing circles}. A \emph{fully augmented link} is a link obtained by adding a single crossing circle to every twist region of a twist-reduced diagram, and then removing all crossings that bound a bigon. That is, crossings are removed in pairs. The resulting diagram consists of crossing circles that are perpendicular to the plane of projection and strands that lie on the plane of projection except possibly for single crossings in the neighbourhood of a crossing circle. See \reffig{AugLinkConstruction}, right.

Agol and Thurston studied the geometry of fully augmented links using a decomposition into ideal polyhedra~\cite[Appendix]{Lackenby:AltVolume}. In particular, they show that every fully augmented link admits a decomposition  into two identical totally geodesic polyhedra that determine a circle packing on $\CC$. The result we need is the following.

\begin{proposition}\label{Prop:FALSetup}
A fully augmented link decomposes into the union of two identical ideal polyhedra with the following properties.
\begin{enumerate}
\item Each polyhedron is convex, right-angled, with a checkerboard colouring of its faces, shaded and white. The shaded faces are all ideal triangles, each a subset of a 2-punctured disc bounded by a crossing circle.
\item Each polyhedron is determined by a circle packing on $\RR^2\cup \{\infty\}$, with white faces lifting to planes in $\HH^3$ whose boundaries are given by the circles. Shaded faces lift to planes with boundaries given by the dual circle packing.
\item Embed the ideal polyhedron in $\HH^3$ as a convex right-angled polyhedron. Each ideal vertex projects to a link component, or more precisely, the boundary of a sufficiently small horoball neighbourhood of an ideal vertex projects to a subset of a horospherical torus about a link component. 

  Apply an isometry so that an ideal vertex corresponding to a crossing circle lies at the point at infinity in the upper half space model of $\HH^3$. Then two white faces form parallel vertical planes meeting the point at infinity, with two shaded faces forming perpendicular parallel vertical planes, cutting out a rectangle. Two other white faces, defining circles tangent to the white parallel vertical planes, meet the shaded parallel vertical planes at right angles. This forms a rectangle with two circles; see \reffig{FALCirclePacking}.

  Moreover, none of the four ideal vertices on $\RR^2$ corresponding to the corners of the rectangle project to crossing circles.
  \end{enumerate}
\end{proposition}

\begin{figure}
\includegraphics{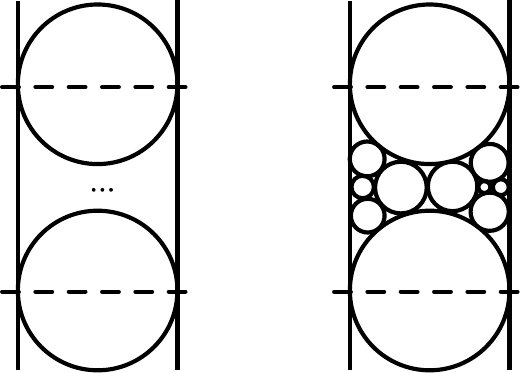}
  \caption{Left: the general form of a circle packing determining a polyhedron, with a vertex that projects to a crossing circle cusp at the point at infinity. The dashed lines show two parallel shaded faces meeting infinity. The region between the circles and lines will be filled with a circle packing. Note white and shaded faces through infinity cut out a rectangle. Right: a specific example.}
  \label{Fig:FALCirclePacking}
\end{figure}

Proofs of \refprop{FALSetup} can be found in various places, including \cite{FuterPurcell, Purcell:FullyAugmented}. We review the details briefly as it is important to our argument.

\begin{proof}[Proof of \refprop{FALSetup}]
There are two types of totally geodesic surface in the complement of a fully augmented link, and these will form the white and shaded faces of the polyhedra. The first totally geodesic surface comes from embedded 2-punctured discs bounded by crossing circles; colour each of these shaded.

The second comes from a surface related to the plane of projection. If the fully augmented link has no crossings on the plane of projection, then it is preserved by a reflection in the plane of projection and the white surface is the plane of projection. This reflection is realised by an isometry fixing the white surface pointwise, so the white surface is totally geodesic. If the link admits single crossings adjacent to crossing circles, reflection in the plane of projection will change the direction of each crossing. However, a full twist about the adjacent crossing circle is a homeomorphism of the complement, and it returns the link to its original position. The combination of reflection followed by twists is an isometry fixing a surface pointwise; this is the white surface. Again it is totally geodesic.

To obtain the decomposition into ideal polyhedra, first cut along each shaded 2-punctured disc. Near single crossings, rotate one copy of the 2-punctured disc by $180^\circ$ to remove the crossing from the diagram. The white face then lies on the projection plane. Slice along the projection plane. This process is shown in \reffig{PolyhedraConstruction}.

\begin{figure}[h]
  \includegraphics{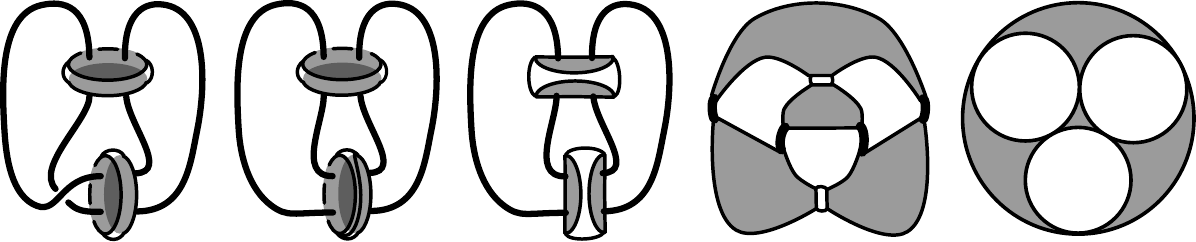}
  \caption{Left to right: Slice along shaded faces bounded by 2-punctured discs and unwind single crossings. Then slice along the white plane of projection. Shrinking remnants of the link to ideal vertices gives the ideal polyhedron. The circle packing is obtained by lifting the polyhedron to $\HH^3$, and taking boundaries of white faces.}
  \label{Fig:PolyhedraConstruction}
\end{figure}

The result is two ideal polyhedra. We now show that these satisfy the properties stated. First, the checkerboard colouring is as claimed, by construction. The involution described above is the reflection through white faces taking one polyhedron to another. Also, note that under the involution, shaded 2-punctured discs are taken to their reflection in the projection plane, hence must still be geodesic. It follows that they are perpendicular to white faces, and so the polyhedron is right-angled. 

The circle packing comes from the totally geodesic white faces. These faces are all disjoint, and correspond to regions of the plane of projection. They lift to a collection of geodesic planes in $\HH^3$, whose boundaries form a collection of circles that are tangent exactly when two white faces are adjacent across a strand of the link, or meet a common crossing circle. The shaded faces lift to ideal triangles, dual to the white circles. Thus this corresponds to a circle packing by shaded circles dual to a circle packing of white circles. The intersections of the exteriors of planes in $\HH^3$ defined by the circles gives a convex region with all right-angled dihedral angles; this is the geometric structure on the ideal polyhedron.

The fact that the cusp is as claimed follows from the fact that each ideal vertex of the polyhedron is 4-valent, so moving one to infinity gives a rectangle, and each shaded face is an ideal triangle, so beneath a vertical shaded face lies a single white circle. The shaded ideal triangle is obtained by slicing a 2-punctured disc through the projection plane. One vertex corresponds to an arc of the crossing circle above (or below) the projection plane, and the other two vertices correspond to strands of the link running through the crossing circle. Thus exactly one of the ideal vertices of the shaded triangle corresponds to a crossing circle. Because the ideal vertex corresponding to the crossing circle lies at infinity, the other two ideal vertices, lying at points of intersection of vertical shaded and white planes, must correspond to link strands on the plane of projection. These are other vertices of the rectangle on $\RR^2$.
\end{proof}

We will show that fully augmented links admit a geometric triangulation coming from the decomposition into polyhedra of \refprop{FALSetup}. To do so, we will show that appropriate neighbourhoods of crossing circles can be triangulated separately. 

Consider a polyhedron of the decomposition of a fully augmented link, as in \refprop{FALSetup}. Arrange the polyhedron in $\HH^3$ so that the point at infinity projects to a crossing circle cusp, with vertical planes cutting out a rectangle on $\bdy\HH^3$. Then there is a unique circle on $\bdy\HH^3$ meeting each vertex of the rectangle. It intersects exactly four of the circles in the circle packing corresponding to white faces of the polyhedron. See \reffig{CutOffCircle}.
The circle is the boundary of a geodesic plane in $\HH^3$.
The intersection of the plane with the polyhedron determines a totally geodesic rectangular surface.

\begin{figure}
  \includegraphics{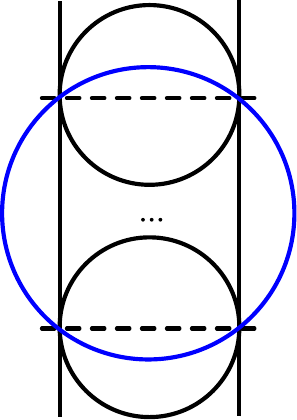}
  \caption{For each ideal vertex corresponding to a crossing circle, there exists a circle running through all four vertices of the rectangle of \refprop{FALSetup}~(3). Each of these defines a plane in $\HH^3$, and their intersection with the polyhedron defines a rectangle $R$.}
  \label{Fig:CutOffCircle}
\end{figure}

\begin{lemma}\label{Lem:CrossingCircleCutoff}
Let $R_1, \dots, R_n$ denote the totally geodesic rectangular surfaces arising as above, one for each crossing circle vertex. Then either the interiors of the rectangles are pairwise disjoint, or there exist exactly two crossing circles, each ideal polyhedron is a regular ideal octahedron, and the rectangle $R_1 = R_2$ cuts each polyhedron into two square prisms. 
\end{lemma}

\begin{proof}
Let $C$ denote the circle running through the four vertices of the rectangle of a crossing circle cusp.
Consider the intersection of this circle $C$ and the circles corresponding to white faces of the polyhedron. The circle $C$ intersects the two vertical planes that form two of the four edges of the rectangle; this gives two intersections. Additionally, the circle $C$ intersects the two hemispheres bounded by circles meeting the vertical planes, at the top and bottom of \reffig{CutOffCircle}.

The fact that $C$ cannot meet any other white faces of the polyhedron now follows from the fact that it encloses the region on $\bdy\HH^3$ bounded by the parallel vertical planes and by the two white circles tangent to them. All other white circles are completely contained in this region. Therefore, the hemispheres they determine cannot intersect $C$.

Now we consider the intersections of two rectangular surfaces $R_1$ and $R_2$ arising from circles $C_1$ and $C_2$ from different crossing circles. Arrange the polyhedron so that the crossing circle vertex corresponding to $R_1$ lies at infinity, and $R_1$ determines a circle $C_1$ running through four vertices of a rectangle.
The rectangular surface $R_2$ lies on a geodesic hemisphere $H_2$ whose boundary on $\bdy \HH^3$ is a circle $C_2$. By the argument above, $C_2$ meets exactly four of the white circles in the circle packing; these intersections cut $C_2$ into four circular arcs. Note that the endpoints of each arc are ideal vertices of the polyhedron.

If the surfaces $R_1$ and $R_2$ intersect, then the circles $C_1$ and $C_2$ intersect. Because the circular arcs of $C_2$ lie inside white circles of the circle packing, this is possible only if a point of intersection of $C_1$ and $C_2$ occurs within one of the circles of the circle packing. Because $C_1$ only meets the two parallel sides of the rectangle and two circles tangent to them, the circles where $C_1$ and $C_2$ intersect must be among these circles.

Next note that $C_1$ and $C_2$ must intersect twice within the same circle of the circle packing, else the ideal vertices met by circular arcs of $C_1$ and $C_2$ interleave on a circle. However, $C_1$ runs through the outermost ideal vertices on each of the four circles under consideration: there are no additional ideal vertices outside the rectangle for $C_2$ to meet for interleaving.

Finally, the intersection points either lie on the vertices of the rectangle defining $R_1$, or lie outside that rectangle, because $C_1$ lies outside that rectangle. But only points inside the rectangle lie inside the polyhedron, so intersections outside the rectangle cannot give intersections of $R_1$ and $R_2$, which both are embedded in the polyhedron.

The only remaining possibility is that $C_1$ and $C_2$ both run through at least two of the same ideal vertices on the rectangle defining $R_1$. If exactly two, $R_1$ and $R_2$ share an edge, but have disjoint interiors as claimed. If more than two, then they must share all four ideal vertices, and $R_1$ and $R_2$ agree. In this case, the polyhedron is determined: it must be a regular ideal octahedron with $R_1=R_2$ cutting off an ideal vertex corresponding to a crossing circle on either side. The fully augmented link can only have two crossings circles, corresponding to the ideal vertices used to define $R_1$ and $R_2$. The rectangles $R_1=R_2$ cut the octahedron into two pyramids over a square base. 
\end{proof}

\begin{proposition}\label{Prop:FALGeomTriang}
Every fully augmented link admits a geometric triangulation with the following properties.
\begin{enumerate}
\item Each crossing circle meets exactly four tetrahedra, two in each polyhedron. 
\item The triangulation is symmetric across the white faces. That is, a reflection across white faces preserves the triangulation. 
\end{enumerate}
\end{proposition}

\begin{proof}
Begin with the ideal polyhedral decomposition of \refprop{FALSetup}. For either one of the two symmetric ideal polyhedra, cut off each ideal vertex corresponding to a crossing circle by cutting along the rectangles of \reflem{CrossingCircleCutoff}. This splits the polyhedron into $n$ pyramids over a rectangular base corresponding to crossing circles, where $n$ is the number of crossing circles, and one remaining convex ideal polyhedron $P$. In the case that there are just two crossing circles, the remaining convex ideal polyhedron is degenerate: it is just the rectangle $R_1=R_2$. In all other cases, it has 3-dimensional interior. 
  
Split each rectangular pyramid into two geometric tetrahedra by choosing a diagonal of the rectangle and cutting along it.

When we reglue into the fully augmented link, the choices of diagonals on the rectangles $R_1, \dots, R_n$ are mapped to ideal edges on the convex polyhedron $P$. When $R_1=R_2$, and $P$ is degenerate, choosing the same diagonal gives the desired triangulation. 

When $P$ is nondegenerate, we triangulate it by coning.
Choose any ideal vertex $w$ of $P$. For any face not containing $w$ that is not already an ideal triangle, subdivide the face into ideal triangles in any way by adding ideal edges. Then take cones from $w$ over all the triangles in faces disjoint from $w$. Because $P$ is convex, the result is a division of $P$ into geometric ideal tetrahedra.

Transfer the triangulation on the first polyhedron to the second by reflection in the white surface. This gives both polyhedra exactly the same subdivision, up to reflection.

Now note that the polyhedra glue by reflection in the white faces, so no new flat tetrahedra need to be introduced in these faces to obtain the gluing. All shaded faces are ideal triangles, which are glued by isometry and again no flat tetrahedra need to be introduced. Thus the decomposition of the polyhedra into geometric tetrahedra described above gives a decomposition of the fully augmented link complement into geometric tetrahedra.

Finally, the fact that the geometric triangulation satisfies the two properties of the theorem follows by construction: Each crossing circle meets two rectangles, hence four tetrahedra. The triangulation is preserved by the reflection in the white faces. 
\end{proof}

The previous lemma gives a geometric triangulation of a fully augmented link, but four tetrahedra meet each crossing circle. We need a triangulation for which only two tetrahedra meet each crossing circle.

\begin{proposition}\label{Prop:FALGeomTriang2Tet}
Every fully augmented link admits a geometric triangulation with the property that each crossing circle meets exactly two tetrahedra.
\end{proposition}

\begin{proof}
Begin with the geometric triangulation of \refprop{FALGeomTriang} and consider the crossing circles. These are triangulated by exactly four tetrahedra, two in each polyhedron. For each of the two tetrahedra in one polyhedron, one face lies on a totally geodesic ideal rectangle coming from \reflem{CrossingCircleCutoff} embedded in $\HH^3$. Call the two rectangles, one for each polyhedron, $R_1$ and $R_2$. 

Adjust the geometric triangulation as follows. At an ideal vertex corresponding to a crossing circle, two polyhedra are glued along a white face. The result of gluing both polyhedra together along such a face is shown in \reffig{FALGeomTriang}. Note rectangles $R_1$ and $R_2$ are glued along an edge. The boundary of the two glued polyhedra forms an even larger rectangle, with boundary the outermost parallel lines in \reffig{FALGeomTriang}, and the dashed lines.

\begin{figure}[h]
\begingroup%
  \makeatletter%
  \providecommand\color[2][]{%
    \errmessage{(Inkscape) Color is used for the text in Inkscape, but the package 'color.sty' is not loaded}%
    \renewcommand\color[2][]{}%
  }%
  \providecommand\transparent[1]{%
    \errmessage{(Inkscape) Transparency is used (non-zero) for the text in Inkscape, but the package 'transparent.sty' is not loaded}%
    \renewcommand\transparent[1]{}%
  }%
  \providecommand\rotatebox[2]{#2}%
  \newcommand*\fsize{\dimexpr\f@size pt\relax}%
  \newcommand*\lineheight[1]{\fontsize{\fsize}{#1\fsize}\selectfont}%
  \ifx\svgwidth\undefined%
    \setlength{\unitlength}{163.97723007bp}%
    \ifx\svgscale\undefined%
      \relax%
    \else%
      \setlength{\unitlength}{\unitlength * \real{\svgscale}}%
    \fi%
  \else%
    \setlength{\unitlength}{\svgwidth}%
  \fi%
  \global\let\svgwidth\undefined%
  \global\let\svgscale\undefined%
  \makeatother%
  \begin{picture}(1,0.90419766)%
    \lineheight{1}%
    \setlength\tabcolsep{0pt}%
    \put(0.16700444,0.43128414){\color[rgb]{0,0,0}\makebox(0,0)[lt]{\lineheight{1.25}\smash{\begin{tabular}[t]{l}$R_1$\end{tabular}}}}%
    \put(0.76879097,0.43359429){\color[rgb]{0,0,0}\makebox(0,0)[lt]{\lineheight{1.25}\smash{\begin{tabular}[t]{l}$R_2$\end{tabular}}}}%
    \put(0,0){\includegraphics[width=\unitlength,page=1]{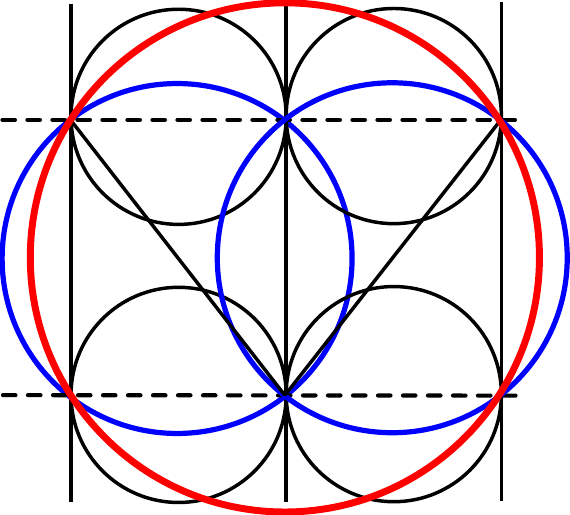}}%
    \put(0.44324936,0.03047806){\color[rgb]{0,0,0}\makebox(0,0)[lt]{\lineheight{1.25}\smash{\begin{tabular}[t]{l}$C$\end{tabular}}}}%
    \put(0,0){\includegraphics[width=\unitlength,page=2]{FALGeomTriang.pdf}}%
  \end{picture}%
\endgroup%

  \caption{Glue both polyhedra in the decomposition of a fully augmented link along a white face corresponding to a vertical plane in $\HH^3$. Two rectangles, $R_1$ and $R_2$, are glued as shown. A larger circle $C$ runs through vertices of both $R_1$ and $R_2$.}
  \label{Fig:FALGeomTriang}
\end{figure}

There is a circle $C$ running through each vertex of that rectangle, shown in red in \reffig{FALGeomTriang}. Note that the hemisphere defined by $C$ in $\HH^3$ meets an edge of $R_1$ and an edge of $R_2$. Cutting along the hemisphere cuts the two polyhedra along a rectangle $R$. The region bounded by $R_1$, $R_2$, $R$ and the two polyhedra is a solid with three ideal quadrilateral faces and two ideal triangle faces; it forms a prism over an ideal triangle. The region between $R$ and infinity forms a neighbourhood of the crossing circle vertex. Triangulate it by adding an edge along a diagonal of $R$ and then coning to infinity. This gives two geometric ideal tetrahedra meeting the crossing circle vertex. These are the only tetrahedra meeting this vertex.

It remains to triangulate the prism over the ideal triangle bounded by $R_1$, $R_2$, and $R$. The rectangles $R_1$, $R_2$, and $R$ all have been triangulated by a choice of diagonal; the one on $R$ comes from the paragraph just above, and those on $R_1$ and $R_2$ come from \refprop{FALGeomTriang}. These three edges determine two ideal triangles whose interiors are disjoint in the interior of the triangular prism. They divide the prism into three geometric tetrahedra.
\end{proof}

\section{Layered solid tori}\label{Sec:LST}

To obtain highly twisted links, we will be performing Dehn filling on the crossing circle cusps of fully augmented links, using the triangulation of \refprop{FALGeomTriang2Tet}. We need a triangulation of the solid torus used in the Dehn filling. The triangulation that will work in this setting is a  \emph{layered solid torus}, first described by Jaco and Rubinstein~\cite{JacoRubinstein:LST}. In this section, we will review the construction of layered solid tori, and how they can be used to triangulate a Dehn filling of a triangulated manifold such as a fully augmented link. 

The boundary of a layered solid torus consists of two ideal triangles whose union is a triangulation of a punctured torus. The space of all such triangulations of punctured tori is described by the Farey graph. 
Gu{\'e}ritaud and Schleimer present a description of the layered solid torus using the combinatorics of the Farey graph~\cite{GueritaudSchleimer}, and then glue this into the boundary of a manifold to be Dehn filled. We will follow their presentation.

\subsection{Review of layered solid tori}
Recall first the construction of the Farey triangulation of $\HH^2$. We
view $\HH^2$ in the disc model, with antipodal points $0/1$ and $1/0=\infty$ in $\bdy\HH^2$ lying on a horizontal line through the centre of the disc, with $1/0$ on the left and $0/1$ on the right. Put $1/1$ at the north pole, and $-1/1$ at the south pole. Two points $a/b$ and $c/d$ in $\QQ\cup \{1/0\}\subset\bdy\HH^2$ have distance measured by
\[ i(a/b, c/d) = |ad - bc|. \]
Here $i(\cdot, \cdot)$ denotes geometric intersection number of slopes on the torus.
We draw an ideal geodesic between each pair $a/b$, $c/d$ with $|ad-bc|=1$. This gives the \emph{Farey triangulation}.

Any triangulation of a once-punctured torus consists of three slopes on the boundary of the torus, with each pair of slopes having geometric intersection number $1$. Denote the slopes by $p$, $q$, $r$. Note that this triple determines a triangle in the Farey triangulation. Moving across an edge of the Farey triangulation changes the triangulation by replacing one slope with another, say $r'$ replaces $r$. See \reffig{Farey}.

\begin{figure}[h]
  \includegraphics{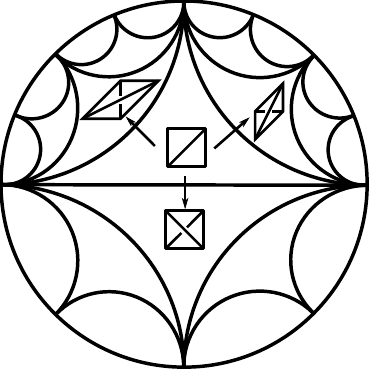}
  \caption{Each triangle in the Farey graph determines a triangulation of a punctured torus. Moving across an edge replaces one of the three slopes of the triangulation by a different slope.}
  \label{Fig:Farey}
\end{figure}

In the case that we wish to perform a Dehn filling by attaching a solid torus to a triangulated once-punctured torus, there are four important slopes involved. Three of the slopes are the slopes of the initial triangulation of the once-punctured solid torus. In our setting, these will typically either be $\{0/1, 1/0, 1/1\}$ or $\{0/1, 1/0,-1/1\}$. They form an \emph{initial triangle} in the Farey graph. The last slope is $m$, the slope of the Dehn filling. 

Now consider the geodesic in $\HH^2$ from the centre of the initial triangle to the slope $m\subset \bdy \HH^2$. This passes through a sequence of triangles in the Farey graph by crossing edges of the Farey triangulation. In particular, there will be a finite sequence of triangles, each determined by three slopes,
\[ (T_0, T_1, \dots, T_N) = ( pqr, pqr', \dots, stm), \]
with initial triangle $T_0$ and final triangle $T_N$ such that $m$ is not a slope of any previous triangle in the sequence. For our purposes, we will require that $N\geq 2$. Thus we do not allow $m$ to be a slope of the initial triangle $T_0$ nor a slope of the three triangles adjacent to $T_0$. 

We build a layered solid torus by stacking a tetrahedron onto a once punctured torus, initially triangulated by the slopes of $T_0$, and replacing one slope with another at each step as we stack. That is, two consecutive once punctured tori always have two slopes in common and two that differ by a diagonal exchange. The diagonal exchange is obtained in three-dimensions by layering a tetrahedron onto a given punctured torus such that the diagonal on one side matches the diagonal to be replaced. In \reffig{Farey}, note that the diagonal exchanges have been drawn in such a way to indicate the tetrahedra. 

For each triangle in the path from $T_0$ to $T_{N-1}$, layer on a tetrahedron, obtaining a collection of tetrahedra homotopy equivalent to $T^2\times [0,1]$. At the $k$-th step, the boundary component $T^2\times\{0\}$ has the triangulation of $T_0$ and that $T^2\times\{1\}$ has the triangulation of $T_k$. Continue until $k=N-1$, obtaining a triangulated complex with boundary consisting of two once-punctured tori, one triangulated by $T_0$ and the other by $T_{N-1}$. Recall that $m$ is a slope of $T_N$ --- notice that we are not adding on a tetrahedron corresponding to $T_N$.

If we stop at $T_{N-1}$ (not $T_N$), then one further diagonal exchange will give the slope $m$. That is, $m$ is not one of the slopes of the triangulation of $T_{N-1}$, but a single diagonal exchange replaces the triangulation $T_{N-1}$ with $T_N$, which is a triangulation consisting of two slopes $s$ and $t$ in common with $T_{N-1}$ and the slope $m$ cutting across a slope $m'$ of $T_{N-1}$.

Recall we are trying to obtain a triangulation of a solid torus for which the slope $m$ is homotopically trivial. To homotopically kill the slope $m$, fold the two triangles of $T_{N-1}$ across the diagonal slope $m'$. Gluing the two triangles on one boundary component of $T^2\times I$ in this manner gives a quotient homeomorphic to a solid torus, with boundary still triangulated by $T_0$. Inside, the slopes $t$ and $s$ are identified.
The slope $m$ has been folded onto itself, meaning it is now homotopically trivial. See \reffig{FoldM}.

\begin{figure}[h]
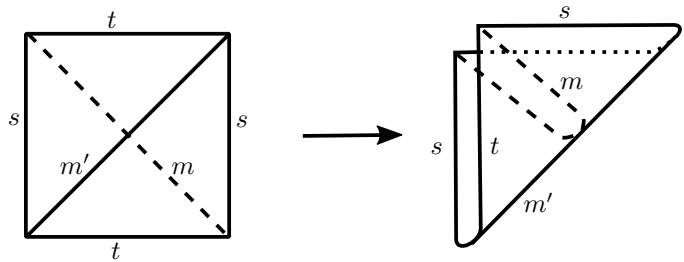
  \caption{Folding $m$ makes it homotopically trivial.}
  \label{Fig:FoldM}
\end{figure}

\section{Angle structures}\label{Sec:AngleStructures}

In order to prove that the triangulations we construct are geometric, we will use tools from the theory of angle structures on 3-manifolds. (These are also often called \emph{strict} angle structures in the literature.)
We are following the lead of Gu{\'e}ritaud and Schleimer in \cite{GueritaudSchleimer}, who use angle structures to show that layered solid tori admit geometric triangulations. The results we need are only slight generalisations of Gu{\'e}ritaud and Schleimer's work, and the proofs follow almost immediately. However, we believe it is useful to step through the results and many of the proofs here as well. Not only does that make this paper more self-contained, but it also sets up a number of tools that we will need later in the paper when we further generalise to different triangulations of solid tori.  Thus in this section we review angle structures, relevant results such as the Casson--Rivin theorem, and we work through the proof that layered solid tori admit geometric triangulations using angle structures.

\begin{definition}\label{Def:AngleStruct}
An \emph{angle structure} on an ideal triangulation $\tau$ of a $3$-manifold $M$ (possibly with boundary) is an assignment of dihedral angles on each tetrahedron such that opposite edges of the tetrahedron carry the same angle, and such that
\begin{itemize}
\item[(1)] all angles lie in the range $(0, \pi)$,
\item[(2)] around each ideal vertex of a tetrahedron, the dihedral angles sum to $\pi$,
\item[(3)] around each edge in the interior of $M$, the dihedral angles sum to $2\pi$. 
\end{itemize}

The set of all angle structures for the triangulation $\tau$ is denoted $\calA(\tau)$. 
\end{definition}

An angle structure on an ideal tetrahedron uniquely determines a hyperbolic structure on that tetrahedron. However, an angle structure on a triangulated 3-manifold is not as restrictive as a geometric triangulation. While one can assemble a space from hyperbolic triangles determined by the angles, under the gluing there may be shearing along edges. Thus the structure does not necessarily give a hyperbolic structure on $M$. 

However, an angle structure determines a volume, by summing the volumes of the hyperbolic ideal tetrahedra with the dihedral angles given by the angle structures. That is, recall that a hyperbolic ideal tetrahedron with dihedral angles $\alpha$, $\beta$, $\gamma$ has volume $\Lambda(\alpha)+\Lambda(\beta)+\Lambda(\gamma)$, where $\Lambda$ is the Lobachevsky function. Define the volume functional $\calV\from \calA(\tau)\to\RR$ as follows. 
For $p \in \calA(\tau)\subset \RR^{3n}$, assign to the angle structure $p = (p_1, p_2, p_3, \dots, p_{3n})$ the real number
\[ \calV(p) = \Lambda(p_1)+\Lambda(p_2) + \Lambda(p_3) + \dots + \Lambda(p_{3n}). \]

The volume functional is a convex function on $\calA(\tau)$. That means it either takes its maximum on the interior of the space $\calA(\tau)$, or there is no maximum in $\calA$, and $\calV$ is maximised on the boundary of the closure $\overline{\calA(\tau)}$. See, for example, \cite{FuterGueritaud:Survey}.

The following theorem, proved independently by Casson and Rivin, will allow us to use angle structures to obtain a geometric triangulation in the case that the maximum occurs in the interior of the space $\calA(\tau)$.

\begin{theorem}[Casson, Rivin]\label{Thm:CassonRivin}
Let $M$ be an orientable $3$-manifold with boundary consisting of tori, and let $\tau$ be an ideal triangulation of $M$. Then a point $p \in \calA(\tau)$ corresponds to a complete hyperbolic metric on the interior of $M$ if and only if $p$ maximises the volume functional $\calV \from \calA(\tau) \to \RR$.
\end{theorem}

The proof of \refthm{CassonRivin} follows from work in \cite{Rivin:EuclidStructs}. A different proof that includes a nice exposition is given by Futer and Gu{\'e}ritaud~\cite{FuterGueritaud:Survey}.

\subsection{Angle structures on layered solid tori}
This subsection is devoted to the following proposition and its proof, which guarantees an angle structure on a layered solid torus. The result is essentially
\cite[Proposition~10]{GueritaudSchleimer}, and the proof is very similar. However, our statement is slightly more general. Additionally, parts of the proof will be needed in a later section, so we include the full argument. 

\begin{proposition}\label{Prop:LSTAngleStructExists}
Let $p, q, r$ be slopes on the torus that bound a triangle in the Farey graph in $\HH^2$. Let $m$ be a slope separated from the triangle $(p,q,r)$ by at least one triangle; that is, the geodesic $\gamma$ in $\HH^2$ from the centre of triangle $(p,q,r)$ to $m$ intersects at least three triangles (one containing $m$, one containing $(p,q,r)$, and at least one more). Relabel $p,q,r$ if necessary so that the geodesic $\gamma$ exits the triangle $(p,q,r)$ by crossing the edge $(p,q)$, and exits the next triangle $(p,q,r')$ by crossing the edge $(q,r')$. (Thus $r$ is the first slope to disappear from the triangulation, and $p$ is the second.) Assign to $p,q,r$ \emph{exterior} dihedral angles $\theta_p, \theta_q, \theta_r$, respectively, satisfying
\begin{equation}\label{Eqn:DihedralAngleProperties}
  \theta_p + \theta_q + \theta_r = \pi, \quad
  -\pi < \theta_p,\theta_q < \pi \quad
  \mbox{ and } \quad 0 < \theta_r < \pi.
\end{equation}
Finally, consider the layered solid torus $T$ with boundary $\bdy T$ a punctured torus triangulated with slopes $p,q,r$, and meridian the slope $m$. Set interior dihedral angles at edges of slope $p$, $q$, $r$ of $\bdy T$ equal to $\pi-\theta_p$, $\pi - \theta_q$, $\pi-\theta_r$, respectively.

There exists an angle structure on $T$ with the given interior dihedral angles if and only if
\begin{equation}\label{Eqn:IntersectionCondition}
  i(m,p)\theta_p + i(m,q)\theta_q + i(m,r)\theta_r > 2\pi,
\end{equation}
where $i(a,b)$ denotes geometric intersection number. 
\end{proposition}

\begin{remark}
Gu{\'e}ritaud and Schleimer actually requre $0 \leq \theta_p,\theta_q$ in the statement of their proposition. However, most of their argument applies equally well if one of $\theta_p,\theta_q$, say $\theta_p$, is negative, provided $\pi-\theta_p<2\pi$. The other conditions will then imply that $\theta_q$ is positive, and that $\theta_p +\theta_r$ is positive, and these conditions suffice to prove the proposition. 
\end{remark}

We start by setting up notation. 
Let $T$ be a layered solid torus constructed by following a geodesic $\gamma$ in the Farey graph in $\HH^2$ from the centre of triangle $(p,q,r)$ to the slope $m$, with $\gamma$ intersecting at least three triangles. Let $(T_0, T_1, \dots, T_{N-1}, T_N)$ denote the sequence of triangles. As in the statement of \refprop{LSTAngleStructExists}, we label such that $r$ is the first slope to be replaced by diagonal exchange and $p$ is the second. 

Let $\Delta_1, \dots, \Delta_{N-1}$ denote the tetrahedra in $T$, constructed as in \refsec{LST}. Thus $\Delta_1$ meets the boundary of the layered solid torus in slopes $(p,q,r)$, and the tetrahedron $\Delta_{N-1}$ is folded on itself to form the solid torus with $m$ homotopically trivial.

\begin{lemma}\label{Lem:HexagonCusp}
The solid torus $T$ has a single ideal vertex. A horosphere about this ideal vertex intersects each tetrahedron of $T$ in four triangles, arranged corner to corner such that their outer boundary forms a hexagon, with opposite angles agreeing. For tetrahedra $\Delta_1, \dots, \Delta_{N-2}$, an inner boundary is also a hexagon, with inner boundary of the triangles of $\Delta_i$ identified to the outer boundary of the triangles of $\Delta_{i+1}$.
For tetrahedron $\Delta_{N-1}$, the four triangles form a solid hexagon.
\end{lemma}

\begin{proof}
Consider the boundary of any layered solid torus. This is a 1-punctured torus triangulated by two triangles. A path that stays on the 1-punctured torus that runs once around the puncture will run over exactly six triangles; these form a hexagon in the cusp neighbourhood of the solid torus. See \reffig{Hexagon}. Stripping the $k$ outermost tetrahedra off a layered solid torus yields a smaller layered solid torus for $k<N-1$; its boundary still forms a hexagon as in \reffig{Hexagon}.

\begin{figure}[h]
  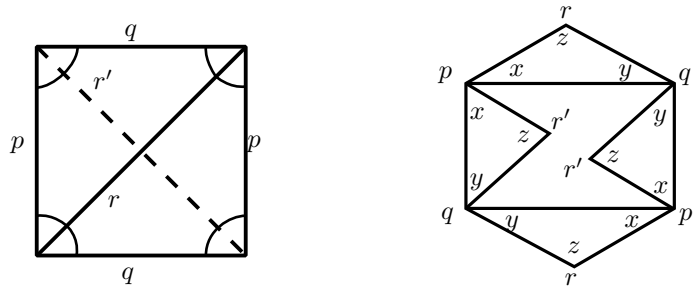
  \caption{Left: a path encircling the puncture of the 1-punctured torus meets exactly six triangles, meeting slopes $p$, $q$, $r$, $p$, $q$, $r$ in order. Right: These triangles lift to give a hexagon in the cusp neighbourhood of a layered solid torus. The tetrahedron that effects the diagonal exchange from $r$ to $r'$ is glued to the hexagon along two faces, forming a new hexagon in the interior.}
  \label{Fig:Hexagon}
\end{figure}

The innermost tetrahedron has its two inside triangles folded together. This gives one of the hexagons shown in \reffig{InnermostHexagon}.\qedhere

\begin{figure}[h]
  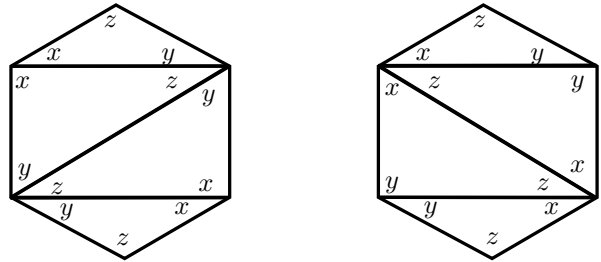
  \caption{The last tetrahedron in the layered solid torus has two interior triangles folded together. The two possible cases are shown.}
  \label{Fig:InnermostHexagon}
\end{figure}
\end{proof}

For the tetrahedron $\Delta_i$, label the (interior) dihedral angles by $x_i$, $y_i$, $z_i$, with $x_i+y_i+z_i=\pi$. By adjusting these labels, we may ensure that $z_i$ is the angle assigned to the slope that is covered by $\Delta_i$, and that $x_i$, $y_i$ are chosen to be in alphabetical order when we run around one of the cusp triangles in anti-clockwise order. These labels agree with the choices in Figures~ \ref{Fig:Hexagon} and~\ref{Fig:InnermostHexagon}. 

Since opposite edges of a tetrahedron have the same angles, this choice of angles $x_i$, $y_i$, $z_i$ completely determines the angles on the hexagons. We summarise the result in the following lemma.

\begin{lemma}\label{Lem:HexagonAngles}
For $i\in \{1, \dots,N-2\}$, two opposite interior angles of the outer hexagon of $\Delta_i$ are $z_i$, two opposite exterior angles of the inner hexagon are $z_i$, and at the four vertices shared by both hexagons, two angles $x_i$ meet at two of the opposite vertices, and two angles $y_i$ meet at two other opposite vertices.

For $\Delta_{N-1}$, the interior angles of the solid hexagon are either $z_{N-1}$, $2x_{N-1}+z_{N-1}$, and $2y_{N-1}$ (with opposite angles agreeing), or $z_{N-1}$, $2y_{N-1}+z_{N-1}$, and $2x_{N-1}$ (with opposite angles agreeing). \qed
\end{lemma}

Gluing tetrahedron $\Delta_{i+1}$ to $\Delta_i$ in the construction of the layered solid torus corresponds to performing a diagonal exchange in the triangulation of the boundary. One of the three edges on the punctured torus boundary is covered by this move. In the cusp picture of the hexagons, gluing $\Delta_{i+1}$ to $\Delta_i$ glues four triangles to the inner hexagon. Two opposite vertices are covered by the triangles and two new vertices are added to the interior. See \reffig{XYinTermsOfZ} for an example. 

The labeling of \reflem{HexagonAngles} implies that two vertices of the inner hexagon formed at the $i$-th step by $\Delta_i$ have interior angle $2\pi-z_i$. These vertices were just added at the previous step by diagonal exchange.  Since the path $\gamma$ in the Farey graph is a geodesic, these vertices will not be covered in the next step. Thus there are two choices for vertices to cover. We call the choices $L$ and $R$, referring to a choice of direction in the Farey graph, as follows. After crossing the first edge in the Farey graph, $L$ and $R$ are determined by the direction the geodesic $\gamma$ takes in the Farey graph, left or right. Except in the last triangle of the Farey graph, this corresponds to attaching a tetrahedron and covering a diagonal. Label the corresponding tetrahedron $\Delta_i$ with an $L$ or $R$, for $i=2, \dots, N-1$; see \reffig{LLFarey} for an example.

\begin{figure}[h]
  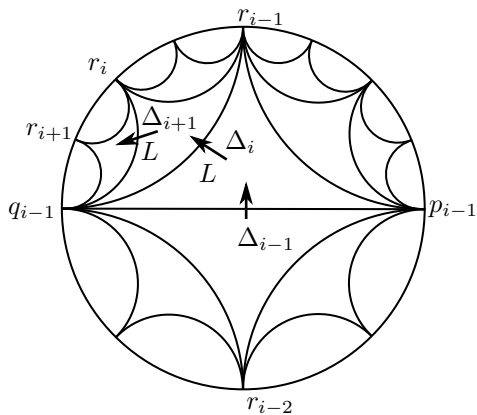
  \caption{Turning left then left again in the Farey graph. Tetrahedron $\Delta_{i-1}$ has inner boundary with slopes $p_{i-1}$, $q_{i-1}$, and $r_{i-1}$, with angles $x_{i-1}$, $y_{i-1}$, and $z_{i-1}$, respectively. Adding $\Delta_i$ in the $L$ direction removes the slope $p_{i-1}$ from the inner boundary, replacing it with slope $r_i$, with angle $z_i$. Adding $\Delta_{i+1}$ removes slope $r_{i-1}$, replacing it with slope $r_{i+1}$, with angle $z_{i+1}$.}
  \label{Fig:LLFarey}
\end{figure}

We need to consider the interior angles of each hexagon. When values of the $z_i$ are given, we will choose the $x_i$ and $y_i$ so that the interior angles form a Euclidean hexagon at each step. 
Consider the outermost hexagon. The slopes of the edges of the outermost hexagon are $p$, $q$, and $r$, and their interior angles are $\pi-\theta_p$, $\pi-\theta_q$, and $\pi-\theta_r$, respectively, as in \refprop{LSTAngleStructExists}. These are chosen such that the sum of all interior angles is $4\pi$, as usual for a Euclidean hexagon. Since tetrahedron $\Delta_1$ covers the edge of slope $r$, the angle $z_1$ must agree with the interior angle along the slope $r$, or $z_1=\pi-\theta_r$. Now we consider the next hexagon.

\begin{lemma}\label{Lem:XYFirstTet}
Let $\theta_r$, $\theta_p$, $\theta_q$ denote exterior dihedral angles as in \refprop{LSTAngleStructExists}. In particular, recall that $r$ is the first slope covered, and $p$ is the second. 

For the first tetrahedron $\Delta_1$, set $z_1=\pi-\theta_r$. Suppose $z_2\in(0,\pi)$, and define a new variable $z_0=\pi+\theta_p$. 

The tetrahedron $\Delta_2$ has either an $L$ or an $R$ label. Assign the same label to $\Delta_1$, so both are labeled $L$ or both are labeled $R$.
\begin{itemize}
\item If both $\Delta_1$ and $\Delta_2$ are labeled $L$, set $x_1 = \pi - (z_0+z_2)/2$, and $y_1 = \pi - z_1 - x_1$.
\item If both $\Delta_1$ and $\Delta_2$ are labeled $R$, set $y_1 = \pi - (z_0+z_2)/2$, and $x_1 = \pi - z_1 - y_1$.
\end{itemize}
Then the values of the interior angles of the hexagon between $\Delta_1$ and $\Delta_2$ are $z_2$ (at the two edges of slope $p$), $2\pi-z_1$, and $z_1-z_2$.
\end{lemma}

\begin{proof}
One of the interior angles is immediate: the angle at the newly added edge of slope $r'$ is $2\pi-z_1$.

If $\Delta_2$ is labeled $L$, then the slope $p$ is given angle $x_1$ in $\Delta_1$, as in \reffig{Hexagon}, right. Otherwise it is given angle $y_1$ in $\Delta_1$, based on our orientation conventions. Assume first that $\Delta_2$ is labeled $L$. 

Before adding $\Delta_1$, the interior angle at the edge of slope $p$ was $\pi-\theta_p$. After adding $\Delta_1$, it decreases by $2x_1$. Thus the interior angle is
\[ \pi-\theta_p - 2x_1 = \pi-\theta_p -2\pi + (z_0+z_2)= z_2. \]

Similarly, after adding $\Delta_1$ the interior angle at the edge of slope $q$ becomes
\begin{align*}
  \pi-\theta_q - 2y_1 & = \pi - (\pi-\theta_p-\theta_r) - 2(\pi-z_1-x_1) \\
  & = \theta_p+\theta_r - 2\pi + 2z_1 + 2\pi - z_0 - z_2 = z_1-z_2.
\end{align*}

Similar equations hold, only switching the roles of $x_1$ and $y_1$, if $\Delta_2$ is labeled $R$.
\end{proof}


We will deal with the last tetrahedron $\Delta_{N-1}$ separately. For the others, we have the following result.

\begin{lemma}\label{Lem:XYinTermsOfZ}
  Let $\theta_p$, $\theta_q$, $\theta_r$ denote the exterior dihedral angles as in \refprop{LSTAngleStructExists}. Suppose $z_0=\pi+\theta_p$, $z_1=\pi-\theta_r$, and $z_2, \dots, z_{N-1}$ lie in $(0,\pi)$.

  For $i=1, \dots, N-2$, assign the angles $x_i$ and $y_i$ as below, with assignments depending on the labels ($L$ or $R$) of $\Delta_i$ and $\Delta_{i+1}$:
\begin{itemize}
\item If $\Delta_i$ and $\Delta_{i+1}$ are both labeled $L$, set
  $x_i = \pi - (z_{i-1}+z_{i+1})/2$, and $y_i=\pi-z_i-x_i$.
\item If $\Delta_i$ and $\Delta_{i+1}$ are both labeled $R$, set
  $y_i = \pi - (z_{i-1}+z_{i+1})/2$, and $x_i=\pi-z_i-y_i$. 
\item If $\Delta_i$ is labeled $L$ and $\Delta_{i+1}$ labeled $R$, set
  $y_i=(z_{i-1}-z_i-z_{i+1})/2$, and $x_i=\pi-z_i-y_i$.
\item If $\Delta_i$ is labeled $R$ and $\Delta_{i+1}$ labeled $L$, set
  $x_i=(z_{i-1}-z_i-z_{i+1})/2$, and $y_i=\pi-z_i-x_i$.
\end{itemize}
Then for $i=1, \dots, N-2$, the hexagon between tetrahedra $\Delta_i$ and $\Delta_{i+1}$ has interior angles $z_{i+1}$, $2\pi-z_i$, and $z_i-z_{i+1}$.

Moreover, for any interior edge obtained by layering tetrahedra $\Delta_1, \dots, \Delta_{N-2}$, the sum of the dihedral angles about that edge is $2\pi$. 
\end{lemma}

\begin{proof}
The proof is by induction. We will show that after layering tetrahedron $\Delta_{i+1}$ onto tetrahedra $\Delta_1, \dots, \Delta_i$, the interior edges of hexagons are as claimed, and the sum of dihedral angles around all interior edges is $2\pi$. 

By \reflem{XYFirstTet}, the interior angles of the hexagon are as claimed when $i=1$. When layering $\Delta_1$ onto the tetrahedra outside of the layered solid torus, there are no interior edges created, so the statement on interior edges is vacuously true. 

Now assume by induction that the interior angles of the hexagon between $\Delta_{i-1}$ and $\Delta_i$ are as claimed in the lemma, and that dihedral angles sum to $2\pi$ around any interior edges in the layering of tetrahedra $\Delta_1, \dots, \Delta_i$. Consider $\Delta_{i+1}$.

The argument is mainly a matter of bookkeeping, particularly keeping track of labels on tetrahedra when turning left or right. We have illustrated the process carefully for the case that $\Delta_i$ and $\Delta_{i+1}$ are both labeled $L$. Figure~\ref{Fig:LLFarey} shows the path in the Farey graph. What is important at each step is which slope is covered by the diagonal exchange effected by adding the next tetrahedron. Thus $\Delta_i$ covers a slope $p_{i-1}$ and $\Delta_{i+1}$ covers a slope $r_{i-1}$. 

Figure~\ref{Fig:XYinTermsOfZ} left, shows the effect on the cusp triangulation. In that figure, the outermost hexagon lies on the outside of $\Delta_{i-1}$, with the thick lines the hexagon between $\Delta_{i-1}$ and $\Delta_i$. The edges of $\Delta_{i-1}$ with slopes $r_{i-1}$ are both assigned angle $z_{i-1}$. In the figure, slope $r_{i-1}$ is marked by the red dot.

\begin{figure}[h]
  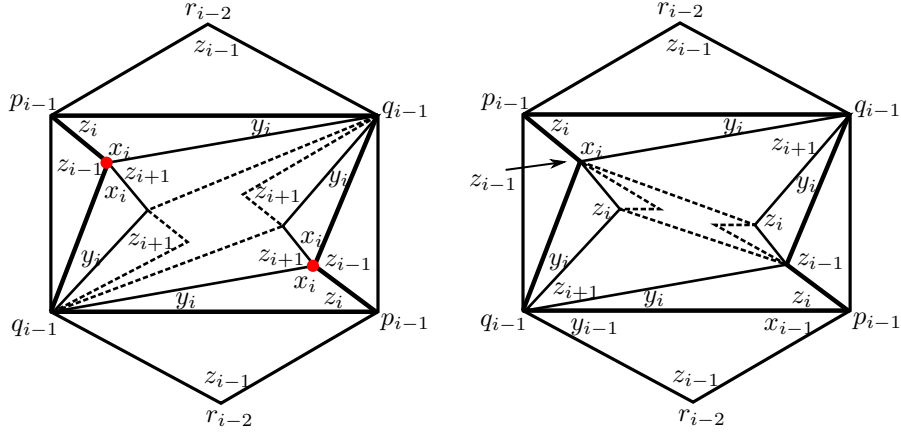
  \caption{Left: The cusp diagram of the portion of a layered solid torus obtained by turning left, then left again. The red dot indicates an edge of the triangulation that is surrounded by the three tetrahedra $\Delta_{i-1}$, $\Delta_i$, $\Delta_{i+1}$. Right: Turning left then right. The vertices of the outer hexagon for $\Delta_{i-1}$ are adjacent to these three tetrahedra, and to no other interior tetrahedra.}
  \label{Fig:XYinTermsOfZ}
\end{figure}

Adding tetrahedron $\Delta_i$ gives a new hexagon, indicated by the thinner line in \reffig{XYinTermsOfZ}, left, between $\Delta_i$ and $\Delta_{i+1}$. The edges of $\Delta_i$ with slopes $p_{i-1}$ and $r_i$ are assigned angle $z_i$. Our orientation convention then ensures that the edge of slope $r_{i-1}$ is assigned angle $x_i$.

Finally we add tetrahedron $\Delta_{i+1}$. This gives a new innermost hexagon, indicated by the dashed lines in \reffig{XYinTermsOfZ}, left. The edge of slope $r_{i-1}$ is assigned angle $z_{i+1}$.

First we consider the interior angles of the hexagon between $\Delta_i$ and $\Delta_{i+1}$. One of these is $2\pi-z_i$, as desired. The other two are obtained by subtracting $2x_i$ and $2y_i$ from interior angles of the hexagon at the previous step. In particular, we have angles
\[ 2\pi - z_{i-1} - 2x_i = 2\pi-z_{i-1}-2\pi +z_{i-1}+z_{i+1}=z_{i+1}, \]
and
\[ z_{i-1}-z_i - 2y_i = z_{i-1}-z_i-2\pi+2z_i+(2\pi-z_{i-1}-z_{i+1})=z_i-z_{i+1},\]
as desired. 

Notice that after adding tetrahedron $\Delta_{i+1}$, the edge of slope $r_{i-1}$ is completely surrounded by tetrahedra $\Delta_{i-1}$, $\Delta_i$, and $\Delta_{i+1}$, and thus it becomes an interior edge. Notice also that this is the only new interior edge obtained by adding $\Delta_{i+1}$. Thus we only need to ensure the sum of dihedral angles about this edge is $2\pi$. We read the dihedral angles off of \reffig{XYinTermsOfZ} left:
\[ z_{i-1} + 2x_i + z_{i+1} = 2\pi.\]
This will hold if and only if $x_i$ satisfies the requirements of the lemma.

A very similar pair of pictures, Farey graph and cusp triangulation, gives the result in the case $\Delta_i$ and $\Delta_{i+1}$ are both labeled $R$. In this case, however, the slope $q_{i-1}$ will be covered by $\Delta_i$. Again $r_{i-1}$ will then be covered by $\Delta_{i+1}$, but by turning right, the angles adjacent to the slope $r_{i-1}$ in this case will be $z_{i-1}$, two copies of $y_i$, and $z_{i+1}$. Thus this case differs from the previous only by switching the roles of $x_i$ and $y_i$. 

If we first turn left then turn right, the slope $p_{i-1}$ is covered first by $\Delta_i$, then $q_{i-1}$ by $\Delta_{i+1}$; see \reffig{XYinTermsOfZ}, right.
The interior angles of the hexagon between $\Delta_i$ and $\Delta_{i+1}$ are $2\pi-z_i$, $2\pi-z_{i-1}-2x_i$, and $z_{i-1}-z_i-2y_i$. The latter two simplify as follows:
\begin{align*}
  2\pi-z_{i-1}-2x_i &  =2\pi-z_{i-1}-2\pi+2z_i+2y_i \\
  &= 2\pi-z_{i-1}-2\pi+2z_i+(z_{i-1}-z_i-z_{i+1}) = z_i-z_{i+1}.
\end{align*}
\[ z_{i-1}-z_i-2y_i = z_{i-1}-z_i - z_{i-1} + z_i + z_{i+1} = z_{i+1}. \]

Finally, in this case, none of the newly added edges are surrounded by the three tetrahedra $\Delta_{i-1}$, $\Delta_i$, and $\Delta_{i+1}$. However, adding $\Delta_{i+1}$ may have created an interior edge at $q_{i-1}$, if $q_{i-1}$ does not lie on the boundary of the layered solid torus. By induction, we know that the interior angle of the hexagon between $\Delta_{i-1}$ and $\Delta_i$ at the edge of slope $q_{i-1}$ must be $z_{i-1}-z_i$. To this we add two angles $y_i$ coming from $\Delta_i$, and one angle $z_{i+1}$ from $\Delta_{i+1}$. 

In particular, the angles will fit into the Euclidean hexagon, and therefore have the correct angle sum, if and only if 
\[ 2y_i + z_{i+1} = z_{i-1}-z_i. \]
This holds if and only if $y_i$ satisfies the requirement of the lemma.

The case of $R$ followed by $L$ is nearly identical, with the roles of $x_i$ and $y_i$ switched. Thus by induction, the result holds for $i=1, \dots, N-2$. 
\end{proof}

\begin{lemma}\label{Lem:XYLastTet}
Consider the last tetrahedron $\Delta_{N-1}$. Assign a label $L$ or $R$ to an empty tetrahedron $\Delta_N$ depending on whether $\gamma$ turns left or right when running into the final triangle $T_N$ of the Farey complex, and set $z_N=0$. Define the angles $x_{N-1}$ and $y_{N-1}$ in terms of $z_N$, $z_{N-1}$, an $z_{N-2}$ depending on the labels $L$ or $R$ on $\Delta_{N-1}$ and $\Delta_N$ exactly as in \reflem{XYinTermsOfZ}. Then the sum of dihedral angles is $2\pi$ around the interior edges in the layered solid torus that are surrounded by $\Delta_{N-2}$ and $\Delta_{N-1}$. 
\end{lemma}

\begin{proof}
The proof is very similar to that of \reflem{XYinTermsOfZ}. The cusp triangulations for cases $LL$ and $LR$ are shown in \reffig{LastTetAngles}.

\begin{figure}[h]
  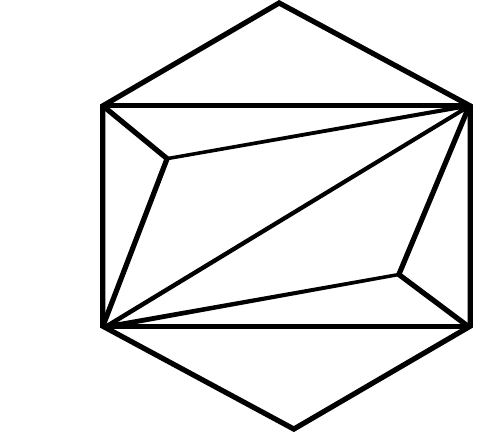
  \hspace{.2in}
  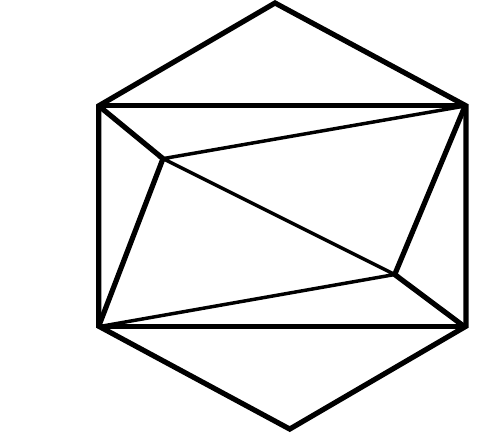
  \caption{Shown are both cases when $\Delta_{N-1}$ is labeled $L$. On the left, the empty tetrahedron $\Delta_N$ is labeled $L$, and on the right, the empty tetrahedrahedron $\Delta_N$ is labeled $R$.}
  \label{Fig:LastTetAngles}
\end{figure}

In the case $LL$, exactly one edge in the interior of the solid torus is surrounded by $\Delta_{N-2}$ and $\Delta_{N-1}$. The sum of the angles around this edge is
\[ z_{N-2} + 2x_{N-1} = z_{N-2} + 2\pi-z_{N-2}-z_N =2\pi,\]
since $z_N=0$. Thus the sum is $2\pi$ in the $LL$ case when $i=N-1$.

In the case $LR$, an interior edge is surrounded by $\Delta_{N-1}$ and $\Delta_{N-2}$, and the sum of angles around the edge must be
\begin{align*}
z_{N-2}+2x_{N-1}+z_{N-1} &= z_{N-2}+z_{N-1}+2(\pi-z_{N-1}-y_{N-1}) \\
& = z_{N-2}+z_{N-1}+2\pi - 2z_{N-1} - z_{N-2} + z_{N-1} + z_N =2\pi.
\end{align*}

The cases $RR$ and $RL$ hold similarly. 
\end{proof}

\begin{lemma}\label{Lem:ZSequence}
Let $\theta_p$, $\theta_q$, $\theta_r$ be as in \refprop{LSTAngleStructExists}. Let
\[(z_0=\pi+\theta_p, z_1=\pi-\theta_r, z_2, \dots, z_{N-1}, z_N=0) \]
be a sequence of numbers with $z_i\in (0,\pi)$ for $i=1, \dots, N-1$.
Let $x_i$ or $y_i$ be defined in terms of the sequence of the $z_j$ via the equations of \reflem{XYinTermsOfZ}. 
Then $x_i, y_i, z_i$ give an angle structure on the layered solid torus if and only if for each $i=1, \dots, N-1$, the sequence satisfies:
\[
\begin{cases}
  z_{i-1}> z_i + z_{i+1} & \mbox{ if $\Delta_i$ and $\Delta_{i+1}$ are labeled $RL$ or $LR$ (hinge condition)} \\
  z_{i-1} + z_{i+1} > 2z_i
  & \mbox{ if $\Delta_i$ and $\Delta_{i+1}$ are labeled  $RR$ or $LL$ (convexity condition)} \\
\end{cases}
\]
and additionally $z_2<\pi-\theta_p$.

Moreover, if they give an angle structure, then the sequence is strictly decreasing.
\end{lemma}

\begin{proof}
Suppose first that we have an angle structure. Then $x_i, y_i, z_i \in (0,\pi)$ for $i=1, \dots, N-1$, and $x_i+y_i+z_i=\pi$. We can use this equation along with the equations of \reflem{XYinTermsOfZ} to write both $x_i$ and $y_i$ in terms of $z_{i-1}$, $z_i$, and $z_{i+1}$. 

In the $LL$ or $RR$ case, each of $x_i$ and $y_i$ are one of
\begin{equation}\label{Eq:XYLLCase} (z_{i-1}-2z_i+z_{i+1})/2 \quad \mbox{and} \quad \pi-(z_{i-1}+z_{i+1})/2.
\end{equation}
Thus, because we are assuming we have an angle structure, we have:
\[ 0 < \frac{z_{i-1}-2z_i+z_{i+1}}{2} < \pi \quad \mbox{and} \quad
0< \pi - \frac{z_{i-1} + z_{i+1}}{2} < \pi. \]
The first inequality on the left implies the convexity equation. When $i=1$, the first inequality on the right implies $z_2<\pi-\theta_p$. 

In the $RL$ or $LR$ case, each of $x_i$ and $y_i$ are one of
\begin{equation}\label{Eq:XYLRCase} (z_{i-1}-z_i-z_{i+1})/2 \quad \mbox{and} \quad  \pi-(z_{i-1}+z_i-z_{i+1})/2.
\end{equation}
Because we have an angle structure,
\[ 0< \frac{z_{i-1}-z_i-z_{i+1}}{2}<\pi \quad \mbox{and} \quad
0< \pi-\frac{z_{i-1}+z_i-z_{i+1}}{2}<\pi. \]
Again the first inequality on the left implies the hinge equation. This concludes one direction of the proof. 

Now suppose for each $i=1, \dots, N-1$, the sequence satisfies the convexity or hinge condition. We check the conditions on an angle structure, \refdef{AngleStruct}. Condition~(2) holds by our definition of $x_i$ and $y_i$: by hypothesis we require $x_i+y_i+z_i=\pi$.

Condition~(3) follows from \reflem{XYinTermsOfZ} and \reflem{XYLastTet}. These lemmas prove that given our definitions of $x_i$ and $y_i$ in terms of $z_{i-1}$, $z_i$, and $z_{i+1}$, the sum of dihedral angles around every interior edge of the layered solid torus is $2\pi$. 

As for condition~(1), by hypothesis each $z_i\in(0,\pi)$ for $i=1, \dots, N-1$. It remains to show that $x_i$ and $y_i$ lie in $(0,\pi)$ for $i=1, \dots, N-1$. 
In the $LL$ or $RR$ case, we have noted that $x_i$ and $y_i$ are as in equation \eqref{Eq:XYLLCase}. Thus we need
\[ 0 < \frac{z_{i-1}-2z_i+z_{i+1}}{2} < \pi \quad \mbox{and} \quad
0< \pi - \frac{z_{i-1} + z_{i+1}}{2} < \pi. \]
These give four inequalities. When $i=2, \dots, N-1$, three of the inequalities are automatically satisfied when $z_{i-1}$, $z_i, z_{i+1} \in (0,\pi)$ or when $i=N-1$ and $z_N=0$. The final inequality holds if and only if $z_{i-1}+z_{i+1}>2z_i$, which is the convexity condition. 

When $i=1$, the inequalities become
\[ 0 < \frac{(\pi+\theta_p)-2(\pi-\theta_r)+z_2}{2} < \pi \quad \mbox{and} \quad
0< \pi - \frac{(\pi+\theta_p)+z_2}{2} < \pi. \]
These give four inequalities, one of which is automatically true for $0 < \pi+\theta_p < \pi$, and the other three all hold if and only if
\[ \pi-\theta_p-2\theta_r < z_2 < \pi-\theta_p. \]

For $2\leq i\leq N-1$ in the $RL$ or $LR$ case, $x_i$ and $y_i$ are as in equation \eqref{Eq:XYLRCase}. 
Thus we require
\[ 0< \frac{z_{i-1}-z_i-z_{i+1}}{2}<\pi \quad \mbox{and} \quad
0< \pi-\frac{z_{i-1}+z_i-z_{i+1}}{2}<\pi. \]
Again this gives four inequalities, two of which are automatically satisfied for $z_{i-1}, z_i, z_{i+1}\in (0,\pi)$, or when $i=N-1$, for $z_N=0$. The other two inequalities that must be satisfied are $z_{i-1}>z_{i+1}-z_i$ and $z_{i-1}>z_{i+1}+z_i$. Both hold if and only if $z_{i-1}>z_{i+1}+z_i$.

This proves the if and only if statement of the lemma. 

Now suppose we have an angle structure. At this point, we know all the inequalities of the lemma must hold, plus an extra one: $z_{i-1} > z_{i+1}$. However, the hinge and convexity equations imply that the sequence is strictly decreasing: The proof is by a downward induction starting at $z_N=0$. This finishes the lemma. 
\end{proof}

\begin{lemma}\label{Lem:NoHingeCase}
  Suppose all tetrahedra are glued via $RR$ or $LL$ and never a hinge $RL$ or $LR$. Then there exists a sequence satisfying the previous lemma if and only if
  \[ i(m,p)\theta_p + i(m,q)\theta_q + i(m,r)\theta_r > 2\pi. \]
\end{lemma}

\begin{proof}
Suppose first that such a sequence holds. 

We claim the convexity condition implies that $z_{N-k} < z_{N-(k+1)} k/(k+1)$ for $k=1, \dots, N-1$. This can be seen by induction: when $k=1$, $z_{N-2}+z_N > 2z_{N-1}$ implies $z_{N-1}<z_{N-2}/2$. Assuming $z_{N-(j-1)}<z_{N-j}(j-1)/j$ then
$z_{N-(j-1)}+z_{N-(j+1)} > 2z_{N-j}$ implies $z_{N-j}(j-1)/j + z_{N-(j+1)} > 2z_{N-j}$, which implies $j z_{N-(j+1)} > (j+1) z_{N-j}$, as desired.

Now observe that when $k=N-1$, the inequality is $N z_1 < (N-1) z_0$, which then becomes $N(\pi-\theta_r) < (N-1)(\pi+\theta_p)$. Simplifying, we obtain
\[ 
\pi < (N-1)\theta_p + N\theta_r \iff
2\pi < \theta_q + N\theta_p + (N+1)\theta_r, \]
using $\theta_p+\theta_q+\theta_r=\pi$.

Suppose that the tetrahedra are all glued in the pattern $LL\dots L$.
Apply an isometry to $\HH^2$ so that the triangle $(p,q,r)$ maps to $(0,1/0,-1)$. Then the slope $m$ is mapped to the slope $N/1\in \QQ$, and the geometric intersection numbers satisfy $i(1/0,N)=1$, $i(0,N)=N$, and $i(-1,N)=N+1$. Because applying an isometry of $\HH^2$ preserves intersection numbers, it follows that the inequality holds above if and only if
\[ i(m,p)\theta_p + i(m,q)\theta_q + i(m,r)\theta_r > 2\pi. \]
The argument in the case that all tetrahedra are glued in the pattern $RR\dots R$ is similar. It follows that if a sequence $(z_0, \dots, z_N)$ exists, then the inequality holds.

Conversely, suppose the inequality holds. Set $z_0=\pi+\theta_p$, $z_1=\pi-\theta_r$. Choose $z_2$ such that
\[\max\{0,2z_1-z_0=\pi-2\theta_r-\theta_p\} < z_2<\min\{z_1=\pi-\theta_r, \pi-\theta_p\}. \]
Inductively, choose a decreasing sequence $z_k$ such that $z_k>2z_{k-1}-z_{k-2}$ and $z_k\in(0,\pi)$. We need to ensure we can choose the sequence all the way to $z_{N-1}$ and set $z_N=0$. Note by this choice of $z_k$, we have
\[ z_k > \pi - k\theta_r - (k-1)\theta_p,\]
so when $k=N$,
\[ z_N > \pi-N\theta_r - (N-1)\theta_p.\]
But as above, the inequality on $\theta_p,\theta_q,\theta_r$ is equivalent to
\[ \pi -N \theta_r - (N-1) \theta_p < 0.\]
Thus we may set $z_N=0$ and satisfy all the required conditions. 
\end{proof}

\begin{lemma}\label{Lem:HingeImpliesInequality}
Suppose there exists a hinge $RL$ or $LR$ in the sequence of labels of $\Delta_1, \dots, \Delta_N$. Then the inequality
\[ i(m,p)\theta_p + i(m,q)\theta_q + i(m,r)\theta_r >2\pi \]
is satisfied for every $\theta_p,\theta_q,\theta_r$ as in \refprop{LSTAngleStructExists}. 
\end{lemma}

\begin{proof}
If there exists a hinge, it is not in the first two labels by choice of $p$ and $r$. Suppose first that the first two labels are $RR$. Apply an isometry to $\HH^2$ taking $(p,q,r)$ to $(0, 1/0, -1)$. Then the first two steps in the Farey graph move from triangle $(0,1/0,-1)$ to $(0,1,1/2)$. There may be some additional number of $R$s in the sequence. Starting at $(0, 1/0, -1)$ and stepping through $n$ initial labels $R$ in the Farey graph puts $\gamma$ in the triangle $(0,1/(n-1), 1/n)$. At this point, the path $\gamma$ goes left,
crossing the edge $(1/(n-1), 1/n)$. Because $\gamma$ never returns to an edge, this means that the slope $m$ lies between $1/(n-1)$ and $1/n$ in the Farey complex. Write $m=a/b$ in lowest terms. The set of rational numbers between $1/(n-1)$ and $1/n$ in the Farey complex can be obtained inductively by summing numerators and denominators of $1/(n-1)$, $1/n$ and other rationals obtained in this manner. Since $a/b$ lies in this range, $a \geq 2$ and $b \geq 2n-1 >2$.

Now, note that for $(p,q,r)=(0,1/0,-1)$, $i(a/b,p)=a$, $i(a/b,q)=b$, and $i(a/b,r)=a+b$. Thus
\[ i(m,p)\theta_p + i(m,q)\theta_q + i(m,r)\theta_r =
a(\theta_p+\theta_r) + b(\theta_q+\theta_r).\]
Because $\theta_p+\theta_q+\theta_r=\pi$, and $-\pi<\theta_p,\theta_q<\pi$, both $\theta_p+\theta_r=\pi-\theta_q$ and $\theta_q+\theta_r=\pi-\theta_p$ are positive. Thus
\[ a(\theta_p+\theta_r) + b(\theta_q+\theta_r) \geq \min\{a,b\}(\theta_p+\theta_q+2\theta_r) = \min\{a,b\}(\pi+\theta_r) > 2\pi. \]
Since intersection numbers are unchanged under isometry of $\HH^2$, this proves the result when the first two labels are $RR$. 

The case that the first two labels are $LL$ is similar.
\end{proof}

\begin{lemma}\label{Lem:AngleStructHinge}
Suppose there exists a hinge $RL$ or $LR$. Then there exists a sequence satisfying \reflem{ZSequence}.
\end{lemma}

\begin{proof}
Let $h \in \{2, 3, \dots, N-1 \}$ be the smallest index such that $\Delta_h$ is a hinge of the form $RL$ or $LR$.  Set $z_0=\pi+\theta_p$, $z_1=\pi-\theta_r$.  We can choose inductively a positive decreasing sequence $z_k$ such that $z_2<\pi-\theta_p$, each $z_k\in(0,\pi)$, and $z_k>2z_{k-1}-z_{k-2}$ for $2 \leq k \leq h$.

The rest of the sequence $z_i$ is constructed backwards from $i=N$ to $i=h$. Consider a sequence $z'_i$. Set $z'_N=0$ and $z'_{N-1}=1$. For each $i$ such that $N-2 \geq i \geq h+1$, inductively choose $z'_i$ such that
$z'_i > z'_{i+1} + z'_{i+2}$ or
$z'_{i} + z'_{i+2} > 2z'_{i+1}$, depending on whether $\Delta_{i+1}$ has a different label ($L$ or $R$) from $\Delta_i$ or not, respectively. Observe $z_i'$ must be greater than $z'_{i+1}$ for each $i$.

Choose $\epsilon$ such that
\[ 0 < \epsilon < \dfrac{z_{h-1}-z_h}{z'_{h+1}}. \]
Set $z_i = \epsilon z'_i$ for $h+1 \leq i \leq N$. We need $z_h$ to satisfy the hinge condition $z_h < z_{h-1} - z_{h+1}$, or $z_h < z_{h-1} - \epsilon z'_{h+1}$. This holds by our choice of $\epsilon$.

Finally, we need each $z_i$ to lie in $(0,\pi)$, for $h+1\leq i \leq N-1$. Observe that $z_{h-1}<\pi$, so $0<\epsilon<\pi/z'_{h+1}$. Then $z_i=\epsilon z_i' < \pi z_i'/z'_{h+1}$. For $h+1 \leq i \leq N-1$, we know $z_i' \leq z'_{h+1}$, hence $z_i'$ is strictly less than $\pi$, as desired. Because $z_i'$ is at least $z'_{N-1}>0$, $z_i=\epsilon z_i'>0$. 
Thus we have found a sequence satisfying \reflem{ZSequence}.
\end{proof}

\begin{proof}[Proof of \refprop{LSTAngleStructExists}]
Suppose
\[ i(m,p)\theta_p + i(m,q)\theta_q + i(m,r)\theta_r \leq 2\pi. \]
By \reflem{HingeImpliesInequality}, there is no hinge $RL$ or $LR$ in the sequence of labels of tetrahedra making the layered solid torus. By \reflem{NoHingeCase}, there does not exist a sequence satisfying \reflem{ZSequence}. But such a sequence is required in an angle structure on a layered solid torus, so there is no angle structure in this case.

Now suppose $i(m,p)\theta_p + i(m,q)\theta_q + i(m,r)\theta_r > 2\pi$.
Then \reflem{AngleStructHinge} and \reflem{NoHingeCase} imply there exists a sequence satisfying \reflem{ZSequence}. It follows from that lemma that there exists an angle structure.
\end{proof}


\subsection{Volume maximisation}

We now show that the volume functional on the space of angle structures takes its maximum on the interior. This is essentially \cite[Proposition~15]{GueritaudSchleimer}, but we extract slightly more information from the proof.

\begin{lemma}\label{Lem:FlatTetrahedra}
Suppose the volume functional on the space of angle structures on a layered solid torus takes its maximum on the boundary. Then the corresponding structure consists only of flat tetrahedra, with angles $(x_i, y_i, z_i)$ a permutation of $(0,0,\pi)$ for each $i=1, \dots, N-1$.
\end{lemma}

\begin{proof}
By work of Rivin~\cite{Rivin:EuclidStructs}, if the volume functional takes its maximum on the boundary of the space of angle structures, then any tetrahedron with an angle $0$ must also have an angle $\pi$. Thus those tetrahedra that do not have all angles strictly within $(0,\pi)$ must have angles $(x_i,y_i,z_i)$ a permutation of $(0,0,\pi)$; this is a flat tetrahedron. 

By \reflem{ZSequence}, a point on the boundary of the space of angle structures corresponds to a sequence $(z_0=\pi+\theta_p, z_1=\pi-\theta_r, z_2, \dots, z_{N-1},z_N=0)$ satisfying the hinge and convexity equations, except the strict inequalities will be replaced by weak inequalities. This must be a nonincreasing sequence.

Suppose the $i$-th tetrahedron is the first flat tetrahedron. Then $z_i\in \{0,\pi\}$ but $z_{i-1}\in (0,\pi)$ unless $i-1=0$. If $z_i=\pi$, then convexity implies $z_j=\pi$ for $j=i+1, \dots, h$, where $h$ is the next hinge index. The hinge condition then implies that all later $z_j=0$. Similarly, if $z_i=0$ then all later $z_j=0$.

Now consider $z_{i-1}$. We have $x_i, y_i, z_i, z_{i+1} \in \{0,\pi\}$. Thus by one of Lemmas~\ref{Lem:XYinTermsOfZ}, \ref{Lem:XYFirstTet}, or \ref{Lem:XYLastTet}, depending on the index $i$, we have $z_{i-1}=2\pi$. But $0< z_{i-1}< \pi$ unless $i-1=0$. So $i-1=0$. Then the first flat tetrahedron is the first tetrahedron, so the entire solid torus consists of flat tetrahedra. 
\end{proof}
  
\begin{corollary}\label{Cor:LSTVolumeMaximum}
Suppose the set of angle structures as in \refprop{LSTAngleStructExists} is nonempty. Then the volume functional takes its maximum on the interior of such angle structures. \qed
\end{corollary}

The following follows immediately from the Casson--Rivin theorem, \refthm{CassonRivin}.

\begin{corollary}\label{Cor:LSTHypStructExists}
For slopes $p,q,r,$ and $m$ as in \refprop{LSTAngleStructExists}, and any angles $\theta_p,\theta_q,\theta_r$ satisfying \eqref{Eqn:DihedralAngleProperties}, there exists a geometric triangulation of the layered solid torus $T$ of that proposition with exterior dihedral angles $\theta_p,\theta_q,\theta_r$.
\end{corollary}

\section{Dehn filling}\label{Sec:DehnFilling}

In this section, we complete the proof of \refthm{HighlyTwisted}.

Let $s$ be a slope, and let $\len(s)$ denote the Euclidean length of a geodesic representative of $s$ on a horospherical cusp torus.

The following theorem is a consequence of Thurston's hyperbolic Dehn filling theorem~\cite{Thurston:Notes}. The version below can be proved assuming a geometric triangulation exists for $M$, with ideas in Benedetti and Petronio~\cite{BenedettiPetronio}, using methods of Neumann and Zagier~\cite{NeumannZagier}. 

\begin{theorem}[Hyperbolic Dehn filling theorem]\label{Thm:ThurstonHypDF}
Let $M$ be a hyperbolic 3-manifold with a geometric ideal triangulation, such that exactly two ideal tetrahedra, $\Delta$ and $\Delta'$, meet a cusp of $M$. Let $s$ be a slope on this cusp. Then for all but finitely many choices of $s$, the Dehn filled manifold $M(s)$ admits a complete hyperbolic structure, obtained by deforming the triangulation of $M$, and taking the completion of the resulting structure. The tips of the tetrahedra $\Delta, \Delta'$ spin asymptotically along the geodesic core of the filling solid torus of $M(s)$. As $\len(s)$ goes to infinity, the cross-ratios of the tetrahedra of $M(s)$ become uniformly close to those of $M$.  \qed
\end{theorem}

In particular, since $M$ admits a geometric ideal triangulation, the cross-ratios of its tetrahedra have strictly positive imaginary part. This is an open condition. Thus for $s$ large enough, the triangulation of $M(s)$ also has cross-ratios with strictly positive imaginary part. It follows that the incomplete, spun triangulation of $M(s)$ is built of geometric tetrahedra. However, we are not interested in incomplete triangulations. We will use the incomplete spun triangulation to build a complete geometric ideal triangulation.

\subsection{Dehn filling and spun triangulations}

The following proposition is essentially Proposition~8 of Gu{\'e}ritaud--Schleimer \cite{GueritaudSchleimer}.

\begin{proposition}\label{Prop:SpunSolidTorus}
Let $X$ be a solid torus with $\bdy X$ a punctured torus, with boundary $\bdy X$ triangulated by two ideal triangles. Let $m\subset \bdy X$ be the meridian of $X$. The following are equivalent.
\begin{enumerate}
\item A complete hyperbolic structure on $X$ is obtained by taking the completion of a spun triangulation consisting of two tetrahedra $\Delta$ and $\Delta'$, where one face of $\Delta$ and one face of $\Delta'$ form the two ideal triangles making up $\bdy X$. 
\item The exterior dihedral angles $a$, $b$, and $c$ on the edges of the triangulation satisfy $a,b\in(-\pi,\pi)$, $c\in(0,\pi)$, and $a+b+c=\pi$, and also 
  \[ n_aa + n_bb + n_cc > 2\pi,\]
where $n_a$, $n_b$, $n_c$ denote the number of times the meridian $m\subset \bdy X$ of $X$ crosses the edge with angle $a$, $b$, $c$, respectively. 
\end{enumerate}
Moreover, if the hyperbolic structure exists on $X$, then it is unique. 
\qed
\end{proposition}

\begin{remark}
Proposition~8 of \cite{GueritaudSchleimer} is not quite stated the same as \refprop{SpunSolidTorus}, but an almost identical proof gives the result claimed here. One difference is that in \cite{GueritaudSchleimer}, they restrict to $a,b,c\in[0,\pi)$. However, this restriction is not required for the proof. What is required, if these form exterior dihedral angles of a solid torus as claimed, is that $a+b+c=\pi$ and $a,b,c \in (-\pi,\pi)$. These conditions follow from considering the Euclidean geometry of a horospherical neighbourhood of the puncture on $\bdy X$. Moreover, the condition $a+b+c=\pi$ forces one of $a,b,c$ to be strictly positive; we let this angle be denoted $c$. 
  
Now in the proof of \cite[Propostion~8]{GueritaudSchleimer}, it is shown that angle structures can be put onto $\Delta$ and $\Delta'$ to form the spun triangulation of $X$ if and only if $n_aa+n_bb+n_cc>2\pi$. In the case the inequality holds, it is shown that the volume functional takes its maximum on the interior of the space of such angle structures, meaning there exists a hyperbolic structure, and that structure is unique by the Casson--Rivin theorem, \refthm{CassonRivin}.
\end{remark}

\begin{theorem}\label{Thm:DehnFillingGeneral}
  Let $L$ be a hyperbolic fully augmented link with $n\geq 2$ crossing circles. Then there exist constants $A_1, \dots, A_n$ such that if $M$ is a manifold obtained by Dehn filling the crossing circle cusps of $S^3-L$ along slopes $s_1, \dots, s_n$ whose lengths satisfy $\len(s_i)\geq A_i$ for each $i=1, \dots, n$, then $M$ admits a geometric triangulation. Allowing some collection of $s_i=\infty$, i.e.\ leaving some crossing circle cusps unfilled, also admits a geometric triangulation.
\end{theorem}

\begin{proof}
By \refprop{FALGeomTriang2Tet}, $S^3-L$ admits a geometric ideal triangulation with the property that each crossing circle meets exactly two ideal tetrahedra. By \refthm{ThurstonHypDF}, for any slope sufficiently long, the Dehn filling along that slope is obtained by taking the completion of a spun triangulation consisting of two tetrahedra. In particular, for the $j$-th twist region, there exists $A_j$ such that if the length of the slope is at least $A_j$, then Dehn filling yields a manifold with spun triangulation. This can be repeated sequentially for each crossing circle, giving constants $A_1, \dots, A_n$. 

For each crossing circle, consider the two tetrahedra that spin around the core of the Dehn filled solid torus. These two tetrahedra together form a spun triangulation of a solid torus. By \refprop{SpunSolidTorus}, this torus is unique, and the exterior dihedral angles $a,b,c$ must satisfy $n_aa+n_bb+n_cc>2\pi$, where $n_a$, $n_b$, $n_c$ denote the number of times the meridian meets the edge on the boundary with corresponding dihedral angle. Then \refcor{LSTHypStructExists} implies there exists a corresponding layered solid torus with the same dihedral angles along slopes on the boundary, and the same meridian, with a complete, geometric hyperbolic structure.

Because dihedral angles agree, each spun solid torus can be removed and replaced by the layered solid torus by isometry. The result is a geometric ideal triangulation of the Dehn filling of $S^3-L$. 
\end{proof}

We can now complete the proof of \refthm{HighlyTwisted} from the introduction.

\begin{theorem}\label{Thm:HighlyTwisted}
For every $n\geq 2$, there exists a constant $A_n$ depending on $n$, such that if $K$ is a link in $S^3$ with a prime, twist-reduced diagram with $n$ twist regions, and at least $A_n$ crossings in each twist region, then $S^3-K$ admits a geometric triangulation.
\end{theorem}

\begin{proof}
If $K$ has a prime, twist-reduced diagram with $n\geq 2$ twist regions, then $S^3-K$ is obtained by Dehn filling a hyperbolic fully augmented link $L$ with $n$ crossing circles, where the Dehn filling is along slopes on each crossing circle. Let $m_j$ be the number of crossings in the $j$-th twist region. Then the length of the $j$-th slope is at least $\sqrt{m_j^2+1}$ by \cite[Theorem~3.10]{FuterPurcell}. 

Note that for fixed $n$, there are only finitely many fully augmented links with $n$ crossing circles. Fix one of these fully augmented links; call it $L_k$. By \refthm{DehnFillingGeneral}, there exist constants $A_{k,1}, \dots, A_{k,n}$ such that if the slope on the $j$-th crossing circle of $L_k$ has length at least $A_{k,j}$, for $j=1, \dots, n$, then the Dehn filling admits a geometric triangulation. Consider $A_n = \max\{A_{k,j}\}$, where the maximum is taken over all links $L_k$ with $n$ crossing circles. Then provided the number of crossings in each twist region of $K$ is at least $A_n$, the length of each slope on each crossing circle will be at least $A_n$, which implies the Dehn filling yields a geometric triangulation. 
\end{proof}

\section{Borromean rings and related links}\label{Sec:Borromean}

In the previous section, we completed the proof of \refthm{HighlyTwisted}, which is unfortunately not effective: the constants $A_1, \dots, A_n$ are unknown. In this section, by restricting the fully augmented links we consider, we are able to prove an effective result, giving an explicit family of hyperbolic 3-manifolds with geometric triangulations. This is similar in spirit to section~5 of \cite{GueritaudSchleimer}, in which Gu{\'e}ritaud--Schleimer show a similar result for Dehn filling one cusp of the Whitehead link. We extend first to the Borromean rings, which is a fully augmented link with two crossing circle cusps, and to the two other fully augmented links with exactly two crossing circle cusps.

\begin{figure}
  \includegraphics{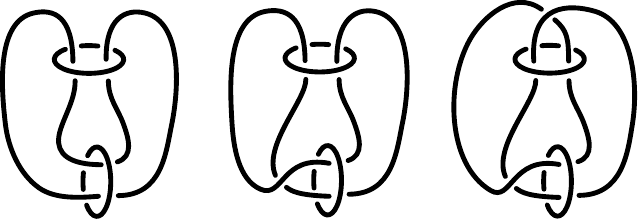}
  \caption{A picture of the Borromean rings as a fully augmented link and the other two fully augmented links with exactly two crossing circles.}
  \label{Fig:BorromeanRings}
\end{figure}

The augmented links we consider next are shown in \reffig{BorromeanRings}. The link on the left of \reffig{BorromeanRings} shows a fully augmented link with three link components; this is ambient isotopic to the Borromean rings. There are two different fully augmented links obtained by inserting half-twists into the crossing circles of the Borromean rings shown; these are the links in the middle and right of that figure. 

Following the procedure for decomposing fully augmented links into polyhedra as in \refsec{FullyAugLinks}, we find that all three links in \reffig{BorromeanRings} decompose into two ideal octahedra; the decomposition for the middle link is exactly the illustration shown in \reffig{PolyhedraConstruction}. 

\begin{lemma}\label{Lem:BorrDecomp}
Let $L$ be one of the three fully augmented links with exactly two crossing circles, as shown in \reffig{BorromeanRings}. Then $M=S^3-L$ has a decomposition into two regular ideal octahedra. Fix one of the two crossing circle cusps. The octahedra meet the fixed crossing circle cusp as follows.
\begin{itemize}
\item One vertex of each octahedron meets the crossing circle cusp. Taking such a vertex to infinity gives a square on $\RR^2$ in $\bdy\HH^3$. We may arrange that one square has corners at $(0,0)$, $(1,0)$, $(1,1)$, and $(0,1)$ in $\RR^2$, and the other has corners at $(1,0)$, $(2,0)$, $(2,1)$, and $(1,1)$ in $\RR^2$.
\item When the crossing circle does not encircle a half-twist, then the arc running from $(0,0)$ to $(0,1)$ projects to a meridian of the crossing circle. When the crossing circle encircles a half-twist (single crossing), the arc from $(0,0)$ to $(1,1)$ projects to a meridian.
\item In all cases, the arc from $(0,0)$ to $(2,0)$ projects to a longitude of the crossing circle, bounding a disc in $S^3$. 
\end{itemize}
\end{lemma}

\begin{figure}[h]
  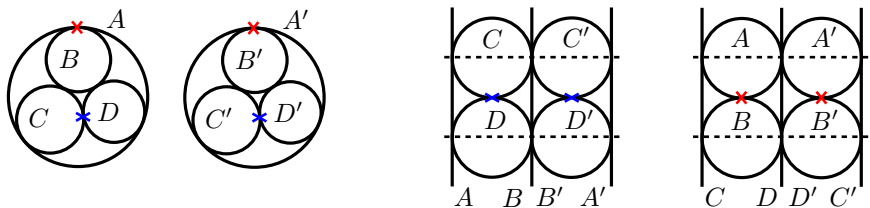
  \caption{On the left are shown the two identical circle packings arising from the decomposition of the links of \reffig{BorromeanRings}. The $x$'s mark the crossing circle cusps in the two polyhedra. On the right, the two cusp neighbourhoods are shown, obtained by taking each pair of $x$'s to infinity. Dashed lines indicate shaded faces.}
  \label{Fig:BorrCirclePacking}
\end{figure}

\begin{proof}
The lemma is proved by considering the decomposition. White faces become the circles shown in \reffig{BorrCirclePacking}, left. The $x$'s on each circle packing indicate crossing circle cusps. Take one of these points to infinity to obtain the required square. Because the two polyhedra are glued along a white side, the squares line up side-by-side as claimed; see \reffig{BorrCirclePacking}, right. A longitude runs along two shaded faces, which runs along the base of both squares, as claimed. 

When there is no half-twist, the shaded face running across the bottom of the square is glued to the shaded face running across the top of the same square, and hence the base of each square is glued to the top of the same square to form a fundamental domain for the cusp torus. A meridian runs along a white side of a square.

When there is a half-twist, a shaded face running across the bottom of the square on the left is glued to the shaded face running across the top of the square on the right, and so a shearing occurs. See \cite[Proposition~3.2]{Purcell:FullyAugmented}.  The result is that a meridian runs across the diagonal of a square, as claimed.
\end{proof}

\begin{lemma}\label{Lem:BorromeanTet}
Let $M$ be the complement of one of the three fully augmented links with exactly two crossing circles. Then $M$ has a decomposition into exactly eight ideal tetrahedra, with four tetrahedra meeting each crossing circle cusp, two in each square of \reflem{BorrDecomp}.

The square bases are glued as follows. The square on the left of one cusp is glued to the square on the left of the other cusp by reflecting across the diagonal of negative slope. The other square, on the right of the first cusp, is glued to the square on the right of the other cusp by reflecting across the diagonal of positive slope, as shown in \reffig{BorrGluing}.
\end{lemma}

\begin{figure}
  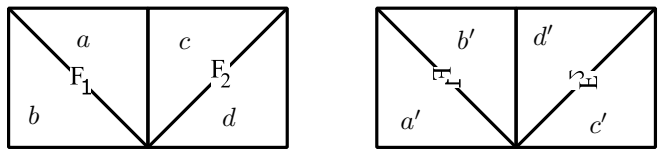
  \caption{Shows how to glue the bases of the crossing circle cusps of fully augmented links with two crossing circles. Note we could choose the opposite diagonals instead.}
  \label{Fig:BorrGluing}
\end{figure}

\begin{proof}
The gluing is obtained by considering squares at the base of the octahedra in \reffig{BorrCirclePacking}. 
For the cusp shown in the middle-right of \reffig{BorrCirclePacking}, one square base has as vertices the points of intersection of circles $A\cap C$, $A\cap D$, $B\cap D$ and $B\cap C$ in anti-clockwise order. This is glued to a square in the opposite cusp meeting the same points of intersection: Note that the points $A\cap C$, $A\cap D$, $B\cap D$ and $B\cap C$ are now in clockwise order in the cusp on the right, with $A\cap C$ and $B\cap D$ in the same location in both. It follows that the squares are glued by a reflection in the negative diagonal.

The other square on the middle-right of \reffig{BorrCirclePacking} has as vertices the points of intersection $A'\cap C'$, $B'\cap C'$, $B'\cap D'$ and $A'\cap D'$ in anti-clockwise order. It is glued to a square on the right with the same vertices, but now $A'\cap C'$, $B'\cap C'$, $B'\cap D'$ and $A'\cap D'$ are in clockwise order on the right, with $A'\cap C'$ and $B'\cap D'$ in the same location. Thus the squares are glued by a reflection in the positive diagonal. 

To triangulate: choose both positive diagonals or both negative diagonals in the cusp on the middle-right of \reffig{BorrCirclePacking}. This splits the two squares into four triangles; either choice of diagonal will do, but we will choose the same diagonal in each square (as opposed to \reffig{BorrGluing} where different diagonals are marked). There are four tetrahedra lying over the four triangles in this cusp. Under the gluing, the diagonals and the squares are preserved, so the four triangles are mapped to four triangles in the second cusp. The four additional tetrahedra lie over these four triangles in the second cusp. 
\end{proof}

\section{Doubling layered solid tori}\label{Sec:DoublingLST}

When the boundary of $M$ is a once-punctured torus triangulated by just two ideal triangles, then we may glue a layered solid torus to $\bdy M$ to perform Dehn filling. In the case of the Borromean rings, the boundary of our manifold is a twice-punctured torus triangulated by four ideal triangles, in symmetric pairs. To perform Dehn filling, we need to modify the construction. This modification essentially appears at the end of Gu{\'e}ritaud and Schleimer~\cite{GueritaudSchleimer}, when they consider the Whitehead link. However, the construction applies much more generally than the Whitehead link application, so we walk through it carefully. 

There are two different modifications required, depending on the slope we wish to Dehn fill. Consider the cover $\RR^2$ of the twice punctured torus obtained by putting punctures at integral points $\ZZ^2\subset \RR^2$. Assume first that there are no half-twists, so a meridian $\mu$ of slope $1/0$ lifts to run from $(0,0)$ to $(0,1)$. A longitude $\lambda$ of slope $0/1$ lifts to run from $(0,0)$ to $(2,0)$. 
Then any slope $m=\ell/k = \ell\mu + k\lambda$ on the torus lifts to an arc beginning at $(0,0)$ and ending at $(2k,\ell)$. 

The two modifications depend on whether $\ell$ is even or odd. 
If $\ell$ is odd, the lift of the slope $\ell/k$ will only meet the points of $\ZZ^2$, which are lifts of punctures, at its endpoints.
In this case, we will take a double cover of a layered solid torus.  

\begin{lemma}\label{Lem:DoubleCoverSolidTorus}
Suppose $m=\ell/k = \ell\mu + k\lambda$ is a slope on the torus (with generators $\mu$, $\lambda$ as above) such that $\ell$ is odd, and $\ell/k\notin \{1/0, \pm 1\}$.

Consider first the layered solid torus $X$, constructed as follows. 
Begin in the Farey triangle with vertices $(1/0, 0/1, \pm 1/1)$ and step to the triangle with slope $\ell/2k$, building the corresponding layered solid torus $X$ as in \refsec{LST}.

Let $Y$ be the double cover of $X$. Then $Y$ satisfies the following properties. 
\begin{itemize}
\item The boundary of $Y$ is a twice-punctured torus, triangulated by four ideal triangles (in two symmetric pairs), lifting to give a triangulation of the cover $\RR^2$. The basis slope $\lambda$ lifts to run from $(0,0)$ to $(2,0)$ in $\RR^2$, and projects to run twice around the slope $0/1$ in $\bdy X$. The slope $\mu$ lifts to run from $(0,0)$ to $(0,1)$ in $\RR^2$.
Diagonals of the triangulation of $\bdy Y$ have positive or negative slope, depending on whether $m$ is positive or negative. 
\item The meridian of $Y$ is the slope $m = \ell\mu + k\lambda$.
\end{itemize}
\end{lemma}

\begin{proof}
Let $X$ denote the layered solid torus with a boundary triangulation that includes the slopes $0/1$ and $1/0$ and a diagonal $\pm 1/1$, with sign agreeing with the sign of $m$, and meridian $\ell/2k$.
Note that since $\ell$ is odd and $\ell \notin \{1/0, \pm 1\}$, we have $\ell/2k \notin \{0,1/0, \pm 1, \pm 2, \pm 1/2\}$, which were the excluded slopes for building a layered solid torus in \refsec{LST}.

Let $Y$ denote the double cover of $X$. 
The double cover of a solid torus is a solid torus, and the once-punctured torus boundary lifts to a twice-punctured torus, with triangles lifting to triangles. We need to show that the slopes behave as claimed. 
  
First, the slope $1/0$ and the meridian $\ell/2k$ of $X$ have geometric intersection number $|1\cdot 2k - \ell\cdot 0| = |2k|$, which is even, and thus $1/0$ is homotopic to an even power of the core of the solid torus $X$. Thus in the double cover $Y$, the slope $1/0$ lifts to a closed curve. As $1/0$ is an edge of a triangle on $\bdy X$, it will remain an edge of a triangle on $\bdy Y$, and lift to a generator of the fundamental group denoted $\mu$. We may take this to run from $(0,0)$ to $(0,1)$ in $\RR^2$. 

Next, the curve $0/1$ meets $\ell/2k$ a total of $|0\cdot 2k - \ell\cdot 1|= |\ell|$ times on $\bdy X$, which is odd. Therefore it lifts to an arc rather than a closed curve on $\bdy Y$, with endpoints on distinct punctures. Thus a second generator of the fundamental group of $\bdy Y$ is given by taking two lifts of $0/1$, end to end. Denote this generator by $\lambda$. Its lift runs from $(0,0)$ to $(2,0)$ in $\RR^2$. 

Finally we check that the meridian of $Y$ is the slope $m$, written in terms of $\mu$ and $\lambda$ as claimed. In $X$, the curve $\ell/2k$ bounds a disc. This lifts to bound a disc in $Y$ as well. However, note the lift runs $\ell$ times along $\mu$ and $k$ times along $\lambda$. Thus the meridian slope is as claimed. 
\end{proof}

If $\ell$ is even, say $\ell = 2s$ for some integer $s$, the lift of $m=\ell/k$ to $\RR^2$ is an arc running from $(0,0)$ through $(k, s) \in \ZZ^2$ to $(2k, 2s)\in \ZZ^2$. Thus it meets a lift of a puncture in its interior. In this case, taking a double cover of a layered solid torus will not suffice. Instead, we need to give a different construction. 

\begin{construction}\label{Constr:SideBySide}
Let $m=\ell/k$ be a slope such that $\ell$ is even, say $\ell=2s$, and $m\notin \{0/1, \pm 2/1\}$. 
Let $(T_0, \dots, T_N)$ be a sequence of triangles in the Farey triangulation where $T_0$ is a triangle with slopes $0/1$, $1/0$, and either $1/1$ or $-1/1$, with sign agreeing with the sign of $m$, and $T_N$ is a triangle with slopes $u$, $t$, and $s/k$.

Start with the triangulation of the twice-punctured torus consisting of two copies of the slopes of $T_0$ laying side-by-side. More precisely, fill $\RR^2-\ZZ^2$ with unit squares with diagonals matching that of $T_0$, and quotient by $(x,y) \mapsto (x+2,y)$ and $(x,y)\mapsto (x,y+1)$.

Inductively, for the $j$-th step across an edge in the Farey triangulation, attach two ideal tetrahedra to the twice-punctured torus, effecting two identical diagonal exchanges with the slopes of $T_{j-1}$, and producing a triangulation of a space homotopy equivalent to the product of the interval and the twice-punctured torus, with one boundary triangulated by two side-by-side copies of $T_0$, and the other triangulated by two side-by-side copies of $T_j$. Note that so far, this is identical to the procedure for the layered solid torus, only we are taking two copies of each tetrahedron instead of just one. See \reffig{SideBySide}.

\begin{figure}
  \includegraphics{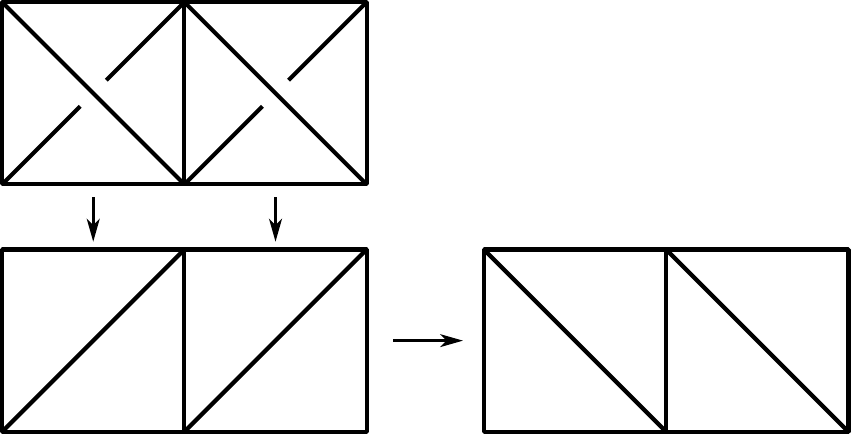}
  \caption{To create a solid torus with boundary a 2-punctured torus, at each step layer two identical tetrahedra onto the current boundary triangulation, effecting a diagonal exchange.}
  \label{Fig:SideBySide}
\end{figure}

This time, continue until $j=N$, so one boundary is labeled by slopes $u$, $t$, and $s/k$, repeated twice in each of two parallelograms lying side-by-side. Now obtain a solid torus as follows. First, identify the two slopes $s/k$ in the two boundary triangles. Then fill the remaining space with a single tetrahedron whose four faces are glued to the inner faces.
\end{construction}

\begin{remark}
Note that to construct a layered solid torus in \refsec{LST}, we needed to exclude slopes in $T_0$ and $T_1$ in the Farey graph. It is no longer necessary to exclude slopes in $T_1$ for the previous construction, because of the addition of extra tetrahedra corresponding to $j=N$ and a final tetrahedron after identifying diagonals of the tetrahedra corresponding to $j=N$. We still exclude slopes in $T_0$. 
\end{remark}

\begin{lemma}\label{Lem:DoubleLST}
Suppose $m=\ell/k = \ell\mu + k\lambda$ is such that $\ell$ is even and $m=\ell/k \notin\{0/1, \pm 2/1\}$. 
Then the triangulated space $Y$ constructed as above, by taking a sequence of side-by-side tetrahedra and attaching a final tetrahedron at the core, forms a solid torus satisfying:
\begin{itemize}
\item The boundary of $Y$ is a twice-punctured torus, triangulated by four triangles (in symmetric pairs), with basis slopes $\mu$ running over one edge of a triangle, lifting to run from $(0,0)$ to $(0,1)$ in $\RR^2$, and $\lambda$ running over two edges (and two punctures), lifting to run from $(0,0)$ to $(2,0)$ in $\RR^2$. Diagonal edges of the triangulation have positive or negative slope, where the sign is determined by the sign of $m$. 
\item The meridian of the solid torus $Y$ is the slope $m = \ell\mu + k\lambda$. 
\end{itemize}
\end{lemma}

\begin{proof}
Generators for the fundamental group of the boundary torus consist of edges $1/0$ in the original triangulation, and two copies of $0/1$ by construction. Denote the first generator by $\mu$ and the second, consisting of two edges, by $\lambda$.

The meridian of the solid torus is the curve that is homotopically trivial. This is the curve formed by pinching together two edges of slope $s/k$ on the inside boundary of triangulated space. Thus it runs twice over this edge. In terms of the generators $\mu$ and $\lambda$, the curve running twice over the edge $s/k$ of the innermost twice-punctured torus has slope $2s\mu + k\lambda$, or $\ell\mu + k\lambda$. 
\end{proof}

\begin{remark}[Symmetry of Construction~\ref{Constr:SideBySide}]
  Notice that the solid torus $Y$ from Construction~\ref{Constr:SideBySide} will admit an involution. This involution takes the innermost tetrahedron to itself (setwise), and for each of the other tetrahedra in the construction, it swaps the two tetrahedra that were layered together, corresponding to the same triangle in the Farey graph. We will give this solid torus an angle structure. We will choose the angles to be preserved under this involution. Thus, although tetrahedra are layered in pairs, and although there are two punctures on the boundary, the two cusp triangulations will be identical, with angles on any tetrahedron agreeing with the angles on its image under the involution. 
\end{remark}

\begin{lemma}\label{Lem:AngleStructDoubleLST}
Let $V$ be a solid torus with twice punctured torus boundary, and with triangulation either as in \reflem{DoubleCoverSolidTorus} or \reflem{DoubleLST}, depending on whether the meridian $m=\ell\mu + k\lambda$ has $\ell$ even or odd. We also assume $m\notin \{0/1, 1/0, \pm 1/1, \pm 2/1 \}$. 

Let $\{\theta_p, \theta_q, \theta_r\}$ be exterior dihedral angles along edges of the twice punctured torus, where each is repeated twice symmetrically, such that
\[
0 < \theta_r < \pi, \quad
-\pi < \theta_p, \theta_q < \pi, \quad \mbox{and} \quad \theta_p+\theta_q+\theta_r = \pi. \]
Suppose also that in the case $\ell$ is odd, intersection numbers satisfy
\[
i(p,\ell/2k)\theta_p + i(q,\ell/2k)\theta_q + i(r, \ell/2k)\theta_r > 2\pi.\]
Then there exists an angle structure on the triangulated solid torus with these exterior angles.
Conversely, if an angle structure exists and $\ell$ is odd, then the intersection numbers satisfy the above equation. 
\end{lemma}

\begin{proof}
The case of the double cover of a layered solid torus follows from the same result for usual layered solid tori, \refprop{LSTAngleStructExists}. In this case, the angle structure exists for the layered solid torus; lift the angles to the double cover to obtain the result. This gives the proof when $\ell$ is odd. 

In the case that $\ell$ is even, we work with the side-by-side solid torus. For every pair of tetrahedra layered on at the $i$-th step of Construction~\ref{Constr:SideBySide}, where $1 \leq i \leq N$, label the dihedral angles of both tetrahedra by $x_i, y_i, z_i$ with $x_i + y_i + z_i = \pi$. Similarly for the last tetrahedron, label its angles $x_{N+1}, y_{N+1}, z_{N+1}$ with $x_{N+1}+y_{N+1}+z_{N+1}=\pi$.  As in \refsec{AngleStructures}, we build an angle structure on the side-by-side layered solid torus by finding a sequence $(z_0 = \pi + \theta_p, z_1 =\pi - \theta_r, \cdots, z_N, z_{N+1})$ with $z_i \in (0, \pi)$.

The side-by-side layered solid torus has two ideal vertices. The cusp triangulation of each of the two ideal vertices is still a sequence of hexagons, and these will be identical because of the symmetry of the solid torus. The cusp triangulations are constructed exactly as in the case of the layered solid torus for the first $N$ steps; these agree with \reffig{XYinTermsOfZ}. But the cusp triangulation differs from that of the usual layered solid torus at the innermost hexagon, where the last tetrahedron is attached. 

Because cusp triangulations agree before the last step, the dihedral angles $x_i, y_i, z_i$ for $2 \leq i \leq N$ satisfy the same conditions of \reflem{XYinTermsOfZ}. Similarly, $z_0=\pi+\theta_p$, $z_1=\pi-\theta_r$, and $x_1$, $y_1$ satisfy the conditions of \reflem{XYFirstTet}, with the same cusp pictures. Therefore, just as in \reflem{ZSequence}, for $1\leq i \leq N$ the sequence satisfies:
\[
\begin{cases}
  z_{i-1}> z_i + z_{i+1} & \mbox{ if $\Delta_i$ and $\Delta_{i+1}$ are labelled $RL$ or $LR$ (hinge condition)} \\
  z_{i-1} + z_{i+1} > 2z_i & \mbox{ if $\Delta_i$, $\Delta_{i+1}$ are $RR$ or $LL$ (convexity condition)} \\
\end{cases}
\]
Moreover, for the first two tetrahedra,
\[
\pi-\theta_p-2\theta_r <z_2 < \pi-\theta_p.
\]

Recall that to make the slope $m$ trivial in the side-by-side case, we identify the two edges of slope $(\ell/2)/k$ in the pair of tetrahedra at the $N$-th step. This corresponds to identifying a pair of opposite vertices in the innermost hexagon; see \reffig{SideBySideHexagon}.

\begin{figure}
  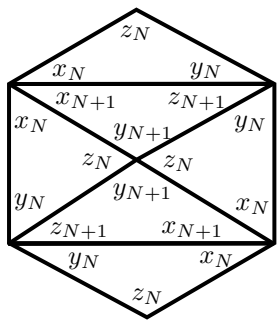
  \caption{The innermost triangles of one of the two hexagons forming a cusp triangulation of a side-by-side solid torus.}
  \label{Fig:SideBySideHexagon}
\end{figure}

As in the proof of \reflem{XYLastTet}, the sum of the interior angles of the hexagon is $4\pi$. This gives the equation
\[ (2x_N + x_{N+1}) + (2y_N+z_{N+1}) + z_N = 2\pi, \mbox{ or } z_N = z_{N+1}+x_{N+1}.\]
Since $x_{N+1}>0$, this implies that $0 < z_{N+1} < z_N$. 
Thus for angle structure to exist on the side-by-side layered solid torus, we require that $0 < z_{N+1} < z_N$.
Finally, note again that by a downward induction, convexity and hinge equations imply that the sequence is strictly decreasing.

As in Lemmas~\ref{Lem:NoHingeCase} and~\ref{Lem:AngleStructHinge}, we find a sequence $(z_0, z_1, \dots, z_N, z_{N+1})$ satisfying the above conditions. In the layered solid torus case, we had to split into two cases, depending on whether or not a hinge existed. When there was no hinge, we needed a convex sequence with the last term equal to zero. However, in this case, we need a convex sequence with $z_{N+1}>0$, which can always be arranged. Thus we only need to show that a sequence satisfying the above requirements exits.

To find such a sequence, the same argument of \reflem{AngleStructHinge} can be used. Namely, let $h \in \{2, 3, \dots, N+1 \}$ be the smallest index such that $\Delta_h$ is a hinge of the form $RL$ or $LR$, or set $h=N+1$ if no such index exists.  Set $z_0=\pi+\theta_p$, $z_1=\pi-\theta_r$.  We can choose inductively a positive decreasing sequence $z_k$ such that $z_k>2z_{k-1}-z_{k-2}$  for $1 \leq k \leq h$.

The rest of the sequence $z_i$ is constructed backwards from $i=N+1$ to $i=h$. Consider a sequence $z'_i$. Set $1>z'_{N+1}>0$ and $z'_{N}=1$. For each $i$ such that $N-1 \geq i \geq h+1$, inductively choose $z'_i$ such that
$z'_i > z'_{i+1} + z'_{i+2}$ or
$z'_{i} + z'_{i+2} > 2z'_{i+1}$, depending on whether $\Delta_{i+1}$ is a hinge or not.

Now choose $\epsilon$ such that
$0 < \epsilon < (z_{h-1}-z_h)/z'_{h+1}$.
Set $z_i = \epsilon z'_i$ for $h+1 \leq i \leq N+1$. The sequence $z_i$ satisfies the required inequalities for $i<h$ and $i>h$. Because of our choice of $\epsilon$, it also satisfies the hinge condition $z_h < z_{h-1} - z_{h+1}$, or $z_h < z_{h-1} - \epsilon z'_{h+1}$. 
Thus we have found a sequence giving an angle structure.
\end{proof}

\begin{lemma}\label{Lem:VolMaxDoubleLST}
Let $V$ be the triangulated solid torus of \reflem{AngleStructDoubleLST}; in particular, the meridian $m$ of $V$ is not one of $\{0/1, 1/0, \pm 1/1, \pm 2/1\}$.
If the volume functional takes its maximum on the boundary of the space of angle structures, then all tetrahedra in the solid torus must be flat. Thus the volume functional is maximised in the interior. 
\end{lemma}

\begin{proof}
Suppose the volume functional is maximised on the boundary for the double cover $V$ of a layered solid torus. There is a symmetry of the triangulated solid torus $V$ swapping triangles in pairs, changing the basepoint of the double covering. If the volume is maximised at a point in which angles are not preserved by this symmetry, then applying the symmetry gives two distinct maxima, contradicting the fact that the volume functional is convex. Thus a point of maximum for the double cover descends to a maximal point for the original layered solid torus. By \reflem{FlatTetrahedra}, if the volume functional is maximised on the boundary, all tetrahedra are flat. Thus the maximum is in the interior in this case. 

For the side-by-side triangulation, a point on the boundary of the space of angle structures corresponds to a sequence $(z_0,z_1,\dots, z_{N+1})$ satisfying hinge and convexity equations, but now with weak inequalities. This must be a non-increasing sequence.
Suppose the $i$-th tetrahedron is the first flat tetrahedron. Then $z_i\in \{0,\pi\}$ but $z_{i-1}\in (0,\pi)$ unless $i-1=0$. If $z_i=\pi$, then convexity implies $z_j=\pi$ for $j=i+1, \dots, h$, where $h$ is the next hinge index. The hinge condition then implies that all later $z_j=0$. Similarly, if $z_i=0$ then all later $z_j=0$.

Now consider $z_{i-1}$. We have $x_i, y_i, z_i, z_{i+1} \in \{0,\pi\}$. Thus by one of the formulas determining angles, as in \reflem{XYinTermsOfZ}, $z_{i-1}=2\pi$. But $0< z_{i-1}< \pi$ unless $i-1=0$. If $i-1=0$, the first flat tetrahedron is the first tetrahedron, so the entire solid torus consists of flat tetrahedra.

There is one final thing to check: above, we have restricted to angle structures in which a tetrahedron and its image under the involution preserving the side-by-side solid torus are given the same angles. We have shown that under this restriction, volume is maximised in the interior. However, note that if the volume were maximised for an angle structure on $Y$ that did not have symmetric angles, then applying the involution would give a distinct angle structure on $Y$ with the same volume, contradicting the fact that the volume functional has a unique maximum. Hence the maximum of the volume functional must occur at an angle structure that is preserved under our involution. 
\end{proof}

\begin{lemma}\label{Lem:ExtAngDoubleLST}
Let $V$ be the triangulated solid torus of \reflem{AngleStructDoubleLST}; in particular, the meridian $m$ of $V$ is not one of $\{0/1, 1/0, \pm 1/1, \pm 2/1\}$. Suppose $V$ is a subset of a triangulation of a 3-manifold $M$ such that the volume functional on $M$ is maximised at a point in which a tetrahedron of $V$ is flat. 
Then the exterior dihedral angles of $V$ satisfy: $(\theta_p,\theta_q,\theta_r)$ are equal to one of $(\pi, -\pi, \pi)$, $(-\pi, \pi, \pi)$, $(\pi,0,0)$, or $(0,\pi,0)$. Conversely, if the exterior dihedral angles satisfies one of these choices, then all tetrahedra in $V$ are flat. 
\end{lemma}

\begin{proof}
By \reflem{VolMaxDoubleLST}, if one of the tetrahedra in the solid torus is flat then all of the tetrahedra in this solid torus are flat.

Next observe that the exterior dihedral angles of the solid torus satisfy:
$\theta_r = \pi - z_1$, $\theta_q = \pi-(z_2 + 2x_1)$ or $\theta_q=\pi-(z_2+2y_1)$, and $\theta_p = \pi - \theta_q - \theta_r$, where $z_1, x_1, y_1$ are the dihedral angles of the outermost tetrahedron in the flat layered solid torus, and $z_2$ is an angle in the next outermost tetrahedron. Exactly one of $x_1, y_1, z_1$ is $\pi$, and the other two are $0$.

If $z_1=0$, then $z_2=0$ because the $z_i$ form a nonincreasing nonnegative sequence. Then $\theta_r=\pi$ and $\theta_q$ is either $\pi$ or $-\pi$, depending on $x_1, y_1$. Since $\theta_p+\theta_q+\theta_r=\pi$, this implies that $(\theta_p,\theta_q,\theta_r) = (\pi, -\pi, \pi)$ or $(-\pi, \pi, \pi)$. 

Now suppose $z_1=\pi$. Then $x_1, y_1=0$, and $z_2$ is $0$ or $\pi$. In this case, $\theta_r = \pi-z_1 = 0$, $\theta_q$ is $0$ or $\pi$ depending on $z_2$, and $\theta_p = \pi-\theta_r  -\theta_q$ is $\pi$ if $\theta_q=0$ or $0$ if $\theta_q=\pi$. Therefore, $(\theta_p,\theta_q,\theta_r) = (\pi, 0, 0)$, or $(0,\pi,0)$.

Conversely, suppose the exterior dihedral angles are as in the lemma. Then $z_1=\pi-\theta_r$ must be $0$ or $\pi$. In either case, the angles of the outermost tetrahedron must then be a permutation of $(0,0,\pi)$, and hence that tetrahedron is flat. By \reflem{VolMaxDoubleLST}, if any tetrahedron is flat, all tetrahedra are flat.
\end{proof}

\subsection{The case of a half-twist}

In the case that a crossing circle encircles a half-twist, the cusp is still formed from two squares, but a meridian of the crossing circle cusp lifts to $\RR^2$ to run from $(0,0)$ to $(1,1)$.

Suppose first that $m$ is positive. Then apply a homeomorphism to the fully augmented link complement that reverses the direction of the crossing. The meridian $\mu = 1/0$ of this new link complement lifts to run from $(0,0)$ to $(-1,1)$, and the longitude $\lambda=0/1$ to run from $(0,0)$ to $(2,0)$. The vertical line from $(0,0)$ to $(0,1)$ is a lift of the slope $1/1$. Now given $\ell$, $k$ relatively prime, perform the construction of the solid torus in \reflem{DoubleCoverSolidTorus} or \reflem{DoubleLST} depending on whether $\ell$ is even or odd. Only now, lift $\mu$ to the curve running from $(0,0)$ to $(-1,1)$ in $\RR^2$. As the lift is purely topological, this gives a solid torus that can be used to perform the Dehn filling just as before. Moreover, Lemmas~\ref{Lem:AngleStructDoubleLST} and~\ref{Lem:VolMaxDoubleLST} still hold with their proofs unchanged in this case.

If $m$ is negative, then the meridian $\mu=1/0$ of the crossing circle cusp lifts to run from $(0,0)$ to $(1,1)$, the longitude $\lambda=0/1$ lifts to run from $(0,0)$ to $(2,0)$, and the slope $-1/1$ lifts to run from $(0,0)$ to $(0,1)$. Given $m=\ell/k$, again perform the construction of the solid torus of \reflem{DoubleCoverSolidTorus} or \reflem{DoubleLST}, depending on whether $\ell$ is even or odd, and lift $\mu$ to the curve running from $(0,0)$ to $(1,1)$ in $\RR^2$. Again \reflem{AngleStructDoubleLST} and \reflem{VolMaxDoubleLST} hold with proofs unchanged to give the required angle structures, with volume maximised in the interior. Thus the construction works equally well with or without a half-twist.

\subsection{A vertical construction}\label{Sec:Vertical}

Finally, the above work is sufficient to perform all Dehn fillings on the family of fully augmented links with exactly two twist regions, shown in \reffig{BorromeanRings}. However, in \refsec{2Bridge}, we will need to extend this construction to obtain a solid torus in which $\mu=1/0$ lifts to run from $(0,0)$ to $(0,2)$ and $\lambda=0/1$ lifts to run from $(0,0)$ to $(1,0)$. We treat that case in this subsection. 

Note that above, we constructed a solid torus by taking side-by-side tetrahedra, i.e.\ stacking identical tetrahedra horizontally, in the $x$-axis direction. More precisely, we tiled all of $\RR^2-\ZZ^2$ by unit squares cut through by a diagonal, layered on tetrahedra coming from a walk in the Farey graph, and then took the quotient by $(2\ZZ, \ZZ)$.

This construction could instead have been done by taking identical tetrahedra stacked in the $y$-axis direction. That is, quotient out by $(\ZZ, 2\ZZ)$. All the results above immediately hold for this construction. In particular, we have the following. 

\begin{lemma}\label{Lem:VerticalSideBySide}
  Let $m=\ell/k$ be such that $k$ is even, and $m\notin\{1/0, \pm 1/2\}$. 
  The ``vertical side-by-side'' solid torus, constructed by layering on tetrahedra in a path from $T_0 =(0/1,\pm 1/1, 1/0)$ to $T_N=(u,v,\ell/(k/2))$, has the following properties.
  \begin{enumerate}
  \item Its boundary is a twice-punctured torus, triangulated by four ideal triangles in two symmetric pairs, with basis slopes $\mu$ running over two edges of a triangle, lifting to run from $(0,0)$ to $(0,2)$ in $\RR^2$, and $\lambda$ running over one edge, lifting to run from $(0,0)$ to $(1,0)$ in $\RR^2$.
  \item The meridian of the solid torus is the slope $m = \ell\mu + k\lambda$.
  \item The triangulated solid torus admits an angle structure, with volume functional taking its maximum in the interior.
  \item In a volume maximising structure, if there is one flat tetrahedron, then all tetrahedra must be flat. Moreover, all tetrahedra are flat if and only if exterior dihedral angles are one of $(\theta_p,\theta_q,\theta_r) = (\pi, -\pi, \pi)$, $(-\pi, \pi, \pi)$, $(\pi,0,0)$, or $(0,\pi,0)$.\qed
  \end{enumerate}
\end{lemma} 

Similarly, if $\ell/k$ is a slope such that $k$ is odd, we may take a double cover of a layered solid torus to produce a solid torus whose boundary is a twice-punctured torus, only now with a fundamental domain that consists of two squares stacked vertically rather than horizontally, as follows.

\begin{lemma}\label{Lem:VerticalDoubleCover}
  Let $m=\ell/k$ be such that $k$ is odd and $m\notin\{0/1, \pm 1/1\}$. Then the (vertical) double cover $Y$ of the layered solid torus $X$ constructed from a walk in the Farey graph from $T_0=(0/1,1/0,\pm 1/1)$ to slope $(2\ell)/k$ has the following properties.
  \begin{itemize}
  \item The boundary of $Y$ is a twice-punctured torus, triangulated by four ideal triangles (in two symmetric pairs), lifting to give a triangulation of the cover $\RR^2$. The basis slope $\mu$ lifts to run from $(0,0)$ to $(0,2)$ in $\RR^2$, and projects to run twice around the slope $1/0$ in $\bdy X$.  The basis slope $\lambda$ lifts to run from $(0,0)$ to $(1,0)$.
  \item The meridian of $Y$ is the slope $m = \ell\mu + k\lambda$.
  \end{itemize}
\end{lemma}

\begin{proof}
Let $X$ be a layered solid torus with meridian slope $2\ell/k$, where $k$ is odd. Let $Y$ be the (vertical double) cover of $X$.

The slope from $(0,0)$ to $(0,1)$ is a generator in the solid torus $X$. It meets the meridian $2\ell/k$ of $X$ a total of $|1 \cdot k - 0 \cdot 2\ell | = |k|$ times, which is odd. Therefore the curve from $(0,0)$ to $(0,1)$ lifts to an arc in $\partial Y$. Thus a generator of the fundamental group of $\partial Y$ is given by taking two lifts of the curve from $(0,0)$ to $(0,1)$, lined up end-to-end. Denote the resulting closed curve in $Y$ by $\mu$. Its lift runs from $(0,0)$ to $(0,2)$ in $\mathbb{R}^2$.

The slope from $(0,0)$ to $(1,0)$ is a generator in the solid torus $X$. This curve and the meridian of $X$, of slope $2\ell/k$, have geometric intersection number $|0 \cdot k - 2 \cdot \ell| = |2\ell|$, which is even. Thus the curve from $(0,0)$ to $(1,0)$ is homotopic to an even power of the core of $X$. Therefore a second generator of the fundamental group of $\partial Y$ is given by taking the lift of  the curve from $(0,0)$ to $(1,0)$. Denote this generator by $\lambda$. Its lift runs from $(0,0)$ to $(1,0)$ in $\mathbb{R}^2$.

The meridian $2\ell/k$ of $X$ lifts to bound a disc in $Y$. Note that the lift runs $2\ell$ times along $\mu$ and $k$ times along $\lambda$. So in the basis for $Y$, the meridian has the form $\ell/k = \ell\mu + k\lambda$. 
\end{proof}

\section{Dehn filling the Borromean rings}\label{Sec:DFBorromean}

In this section we finish the proof that triangulations of Dehn fillings of the crossing circles of the Borromean rings are geometric, for appropriate choices of slopes, and similarly for the other fully augmented link complements shown in \reffig{BorromeanRings}.

\begin{lemma}\label{Lem:TriangBorFilling}
Let $M$ be one of the fully augmented link complements with exactly two crossing circles. 
Let $m_1, m_2$ be slopes such that
\[ m_1, m_2 \notin \{ 0/1, 1/0, \pm 1/1, \pm 2/1 \}. \]
Then the Dehn filling of $M$ on its crossing circle cusps along slopes $m_1$ and $m_2$, denoted $M(m_1,m_2)$, admits a topological triangulation built by gluing together two triangulated solid tori that are either both double covers of layered solid tori, or one double cover and one solid torus with the side-by-side consruction of \reflem{DoubleLST}, or two solid tori of that form.
\end{lemma}

\begin{proof}
The slopes $m_1$ and $m_2$ on the two crossing circle cusps each determine a triangulation of a solid torus by \reflem{DoubleCoverSolidTorus} or \reflem{DoubleLST}. To perform Dehn filling, we remove interiors of all edges, faces, and tetrahedra meeting a crossing circle cusp, and replace them by one of these two triangulated solid tori.

Removing interiors of edges, faces, and tetrahedra meeting the two crossing circle cusps removes all but two squares from the manifold, namely the squares shown in \reffig{BorrGluing}, identified as shown in \reflem{BorromeanTet}. We wish to attach triangulated solid tori to these squares. There are two cases to consider. 

\vspace{.1in}

\textbf{Case 1: The triangulations on the boundary of the solid tori agree.}
If $M$ is the Borromean rings complement, this is the case that the slopes $m_1$ and $m_2$ are both positive, or the case that the slopes $m_1$ and $m_2$ are both negative. In this case, the squares of \reffig{BorrGluing} are triangulated by the same diagonals. If $M$ has one or two half-twists, again the sign of $m_1$ and $m_2$ determine the diagonals. This is the case that the choice of triangulation gives the same diagonals. 

In this case, the corresponding solid tori have boundary triangulation that matches the triangulation on the squares. Moreover, when we attach the two solid tori to opposite sides of the squares, following the gluing instructions of \reffig{BorrGluing}, their triangulations match the given triangulations, giving a topological triangulation of the Dehn filling.

\textbf{Case 2: The triangulations disagree.}
For example, this is the case that $m_1$ and $m_2$ have opposite signs when $M$ is the Borromean rings complement. 

In this case, the triangulations of the solid tori will induce triangulations of the two squares with opposite diagonals. To glue these together, we add two identical extra tetrahedra between the two squares, with the tetrahedra effecting a diagonal exchange on the squares.

Because both tetrahedra are attached along exactly two faces to the second solid torus, in fact they form an extra layer on the solid torus, equivalent to changing the initial triangulation on the boundary in the construction from $(0,1,1/0)$ to $(0,-1,1/0)$, or vice versa. Thus we glue these two tetrahedra to the second solid torus. Then the second solid torus has the form of \reflem{DoubleCoverSolidTorus} or \reflem{DoubleLST}, except with triangulation on its boundary consisting of diagonals having opposite sign from the slope. 
\end{proof}

\begin{lemma}\label{Lem:AnglesBorFilling}
Suppose the exterior dihedral angles on the $i$-th solid torus are denoted $\alpha_i$ for the diagonal edges, and $\theta_i$ for the horizontal edges, for $i=1,2$. If an angle structure exists on $M(m_1,m_2)$ then $\alpha_1=-\alpha_2$ and
\[ \theta_1 = \theta_2+\alpha_2, \quad \theta_2 = \alpha_1+\theta_1. \]
\end{lemma}

\begin{proof}
Angle structures for each of the triangulated solid tori will come from \reflem{AngleStructDoubleLST}, once we decide on exterior dihedral angles. Angle structures on the solid tori will induce an angle structure on the entire manifold if and only if the edge equations are satisfied for edges that lie on the two squares, on the boundaries of the solid tori.

Each solid torus has six edges on its boundary, and we assign three angles, giving a pair of symmetric edges the same angle. On the first solid torus, the horizontal edges all have the same exterior angle, denoted $\theta_1$. The diagonal edges have the same exterior angle, denoted $\alpha_1$, and the vertical edges have the same exterior angle, which must be $\pi-\alpha_1-\theta_1$. Similarly denote the horizontal, diagonal, and vertical exterior angles on the second solid torus by $\theta_2$, $\alpha_2$, and $\pi-\theta_2-\alpha_2$, respectively.

Diagonal edges are glued to diagonal edges. Thus the sum of interior angles satisfies
\[ (\pi-\alpha_1)+(\pi-\alpha_2)= 2\pi, \quad \mbox{or} \quad \alpha_1=-\alpha_2. \]

Horizontal edges with exterior angle $\theta_1$ are identified to vertical edges with exterior angle $\pi-\theta_2-\alpha_2$; both vertical edges in the second solid torus lie in this edge class. Similarly horizontal edges with exterior angle $\theta_2$ are identified to the vertical edges of the first solid torus. Thus the sum of interior angles around these two edge classes satisfy:
\[ 2(\pi-\theta_1) + 2(\theta_2+\alpha_2) = 2\pi, \quad \mbox{or} \quad
\theta_1 = \theta_2 + \alpha_2, \quad \mbox{and} \]
\[ 2(\pi-\theta_2) + 2(\theta_1+\alpha_1) = 2\pi, \quad \mbox{or} \quad
\theta_2 = \theta_1+\alpha_1. \qedhere \]
\end{proof}

\begin{lemma}\label{Lem:DFBorAngleStructExists}
Let $M$ denote the complement of a fully augmented link with two crossing circles, as in \reffig{BorromeanRings}, and let $M(m_1,m_2)$ denote its Dehn filling along slopes $m_1$, $m_2$ on the crossing circle cusps. If $m_1, m_2$ satisfy
\[ m_1, m_2 \notin \{0/1, 1/0, \pm 1/1/, \pm 2/1\}, \]
then $M(m_1, m_2)$ admits an angle structure.
\end{lemma}

\begin{proof}
If an angle structure exists, it must satisfy the equations of \reflem{AnglesBorFilling}. In addition, there is an angle structure on each of two solid tori, with exterior angles denoted by $\theta_{p_1}$, $\theta_{q_1}$, $\theta_{r_1}$ in the first solid torus, satisfying
\begin{equation}\label{Eqn:ThetaEqns}
\theta_{p_1}+\theta_{q_1}+\theta_{r_1}=\pi, \quad -\pi < \theta_{p_1}, \theta_{q_1} < \pi, \quad 0<\theta_{r_1}<\pi,
\end{equation}
and denoted by $\theta_{p_2}$, $\theta_{q_2}$, $\theta_{r_2}$ in the second solid torus, and satisfying similar conditions.

In addition, in the special case that $m_i=\ell_i/k_i$ and $\ell_i$ is odd, and moreover the path in the Farey graph from the initial triangulation has no hinges, then we require the slopes $p_i$, $q_i$, $r_i$ and $\ell_i/(2k_i)$ to satisfy the intersection condition \refeqn{IntersectionCondition}.

Provided we can find any angles that simultaneously satisfy all the above, we will have proved the lemma.

The difficulty is that the angles $\theta_{p_i}$, $\theta_{q_i}$, and $\theta_{r_i}$ have different relationships with angles $\alpha_i$, $\theta_i$ and $\pi-\alpha_i-\theta_i$ depending on whether $m_1,m_2$ have the same or different sign, and on whether a crossing circle has a half-twist.

Suppose first that there are no half-twists, i.e.\ $M$ is the Borromean rings complement, and that $m_1$ and $m_2$ have the same sign; for concreteness, say they are both positive. In this case, we build the two solid tori corresponding to $m_1$ and $m_2$ by starting in the triangle in the Farey graph with vertices $0$, $1$, $1/0$, and go up. In this case, $1$ is not covered in the first step, so $r=0$ or $r=1/0$, and $1$ is the slope $p$ or $q$. If both solid tori have hinges, then the intersection condition will either be automatically satisfied for the layered solid tori we construct, or it is unnecessary for the side-by-side tori we construct. So the more difficult remaining case is when $m_1=\ell_1/k_1$ with $\ell_1$ odd, and there are no hinges in the path from the triangle $(0,1,1/0)$ to $\ell_1/(2k_1)$, and similarly for $\ell_2/(2k_2)$.

No hinges means the slope $\ell_i/(2k_i)$ is of the form $1/2n$ for $n$ a positive integer, or of the form $n/1$; the second is impossible because $1\neq 2k_i$. Thus the path in the Farey graph consists only of $L$s, and goes from $(0,1,1/0)$ to $1/(2k_i)$. Then we know that $r_i$, the first slope covered, corresponds to $1/0$, which is the vertical edge in the initial triangulation, labeled with exterior angle $\pi-\alpha_i-\theta_i$. The slope $p_i$, the second slope covered, corresponds to $1/1$, which is the diagonal edge in the initial triangulation, labeled with exterior angle $\alpha_i$. Thus the slope $q_i$ corresponds to $0$, the horizontal edge, with exterior angle $\theta_i$. The required equations then become: $\theta_i=\theta_{q_i}$, $\alpha_i=\theta_{p_i}$, $\pi-\theta_i-\alpha_i=\theta_{r_i}$ satisfy \refeqn{ThetaEqns}, as well as intersection conditions:
\[
i(1/(2k_i), 1/1)\alpha_i + i(1/(2k_i),0/1)\theta_i + i(1/(2k_i),1/0)(\pi-\alpha_i-\theta_i) > 2\pi, \]
or
\[ (2k_i-1)\alpha_i + \theta_i + 2k_i(\pi-\alpha_i-\theta_i) > 2\pi.\]

There are many solutions to these equations. For example, set $\alpha_1=\pi/6$, $\theta_1=5\pi/9$, and $\pi-\alpha_1-\theta_1=5\pi/18$, and $\alpha_2=-\alpha_1=-\pi/6$, $\theta_2=\alpha_1+\theta_1=13\pi/18$, and $\pi-\alpha_2-\theta_2=4\pi/9$. This gives an angle structure, as desired. 

The case that both $m_1$ and $m_2$ are negative is similar.

Next suppose there are no half-twists, but $m_1$ and $m_2$ have opposite signs; say $m_1>0$ and $m_2<0$. Then we insert an extra tetrahedron onto the first solid torus. Thus the construction of this solid torus now starts at the Farey triangle $(0,-1,1/0)$ and immediately crosses into the triangle $(0,1,1/0)$. It follows that $r=-1$, corresponding to exterior angle $\alpha_1$. The only case in which the intersection condition comes up is if the path in the Farey triangulation only steps $L$, and the slope $m_1 = 1/(2k_1)$. Then in this case, $p=1/0$, corresponding to exterior angle $\pi-\alpha_1-\theta_1$, and $q=0$, corresponding to exterior angle $\theta_1$. The intersection condition is similar to above; the only difference is that we need to ensure we have a solution in which $\alpha_1$ now lies strictly between $0$ and $\pi$. But notice we already found such a solution in the previous case. Thus the same angles in the previous case still work to give an angle structure in this case. Notice that $\alpha_2<0$, but because $m_2<0$ as well, $\alpha_2$ does not correspond to the exterior angle on slope $r_2$, the first slope covered, and so $-\pi < \alpha_2 < 0$ works in this case.

In the case that there is one half-twist, the half-twist changes the names of the slopes in the framing: the meridian of the unfilled manifold is now a diagonal, the longitude runs over two horizontal segments. However, we still assign the same exterior angles $\alpha_1$ to the diagonal, $\theta_1$ to the horizontal. In fact, this gives the same required equations as above, both in the case of $m_1$ and $m_2$ having the same sign, and $m_1$ and $m_2$ having opposite signs, and so the same choices of angles will give an angle structure.

Finally, when there are two half-twists, again we change the framing on both solid tori, but ensuring the triangulations match up will again give the same required equations, and so the solution above always gives an angle structure. 
\end{proof}

\begin{lemma}\label{Lem:VolMaxBor}
Let $M$ denote the complement of one of the fully augmented links with two crossing circles, shown in \reffig{BorromeanRings}, and let $M(m_1,m_2)$ denote its Dehn filling along slopes $m_1$, $m_2$ on the crossing circle cusps. Suppose $m_1, m_2$ satisfy
\[ m_1, m_2 \notin \{0/1, 1/0, \pm 1/1/, \pm 2/1\}. \]
Then, for the space of angle structures on the triangulation of $M(m_1,m_2)$ from above, the volume functional takes a maximum in the interior.
\end{lemma}

\begin{proof}
Consider a point on the boundary of the space of angle structures. Because it is on the boundary, it contains a flat tetrahedron, with angles $0$, $0$, and $\pi$. Because the triangulation of $M(m_1,m_2)$ is built of two triangulated solid tori, one of the tetrahedra in one of the solid tori is flat. Then by \reflem{VolMaxDoubleLST}, all of the tetrahedra in this solid torus are flat. 

By \reflem{ExtAngDoubleLST}, we know that the exterior dihedral angles of the flat solid torus $(\theta_p,\theta_q,\theta_r)$ are either $(\pi, -\pi, \pi)$, $(-\pi, \pi, \pi)$, $(\pi, 0, 0)$, or $(0,\pi,0)$.

Suppose that $(\theta_p,\theta_q,\theta_r) = (\pi, -\pi, \pi)$ or $(-\pi, \pi, \pi)$. In either case, this implies that $\alpha_1 = \pm \pi$, so $-\alpha_2 = \alpha_1 = \pm \pi$, and $\theta_2 = \alpha_1+\theta_1 = 0$ (it cannot be $2\pi$ since we restrict to exterior angles between $-\pi$ and $\pi$), by \reflem{AnglesBorFilling}. Then the third angle satisfies $\pi-\alpha_2-\theta_2 = 0$. So the exterior dihedral angles of the second solid torus are $(\pi, 0,0)$. It follows from \reflem{ExtAngDoubleLST} that the second solid torus must also be flat. Thus such an angle structure has zero volume, and cannot maximise volume.

Now suppose that $(\theta_p,\theta_q,\theta_r) = (0, 0, \pi)$, up to permutation. Then $\alpha_1$ is $0$ or $\pi$, so $\alpha_2 = -\alpha_1$ is $0$ or $-\pi$, and $\theta_2 = \alpha_1+\theta_1$ is $0$ or $\pi$. In any case, the exterior dihedral angles must all be either $0$ or $\pm\pi$, which again implies that the second layered solid torus is flat. As before the angle structure cannot maximise volume. 
\end{proof}

\begin{theorem}\label{Thm:MainBorromean}
Let $L$ be a fully augmented link with exactly two crossing circles, as in \reffig{BorromeanRings}. Let $M$ be the manifold obtained by Dehn filling the crossing circles of $S^3-L$ along slopes $m_1, m_2 \in (\QQ\cup\{1/0\})- \{0, 1/0, \pm 1, \pm 2\}$.
Then $M$ admits a geometric triangulation.
\end{theorem}

\begin{proof}
By \reflem{TriangBorFilling}, $M(m_1,m_2)$ admits a topological triangulation. By \reflem{DFBorAngleStructExists}, this triangulation admits an angle structure. By \reflem{VolMaxBor} the volume functional takes its maximum on the interior of the space of angle structures. Then the Casson--Rivin theorem, \refthm{CassonRivin}, implies that the triangulation is geometric.
\end{proof}

\section{Fully augmented 2-bridge links}\label{Sec:2Bridge}

In this section we consider links obtained by fully augmenting the standard diagram of a 2-bridge link, which we call fully augmented 2-bridge links for short. These admit a decomposition into two identical, totally geodesic, right-angled ideal polyhedra as in \refsec{FullyAugLinks}. In this case, the polyhedra have a particularly nice form: they are built by gluing finitely many regular ideal octahedra. The construction is illustrated carefully in \cite[Section~4]{Purcell:Slope}. We review it briefly here.

A 2-bridge link has one of two forms, depending on whether there is an even or odd number of twist regions; see \reffig{2BridgeForms}. As before, we augment each twist region with a crossing circle, and remove all even pairs of crossings from the corresponding twist region, leaving one or zero crossings encircled by each crossing circle. When there is one crossing, we say the crossing circle has a half-twist. In fact, we will not consider half-twists here, so assume the fully augmented 2-bridge link has no half-twists.

\begin{figure}
\includegraphics{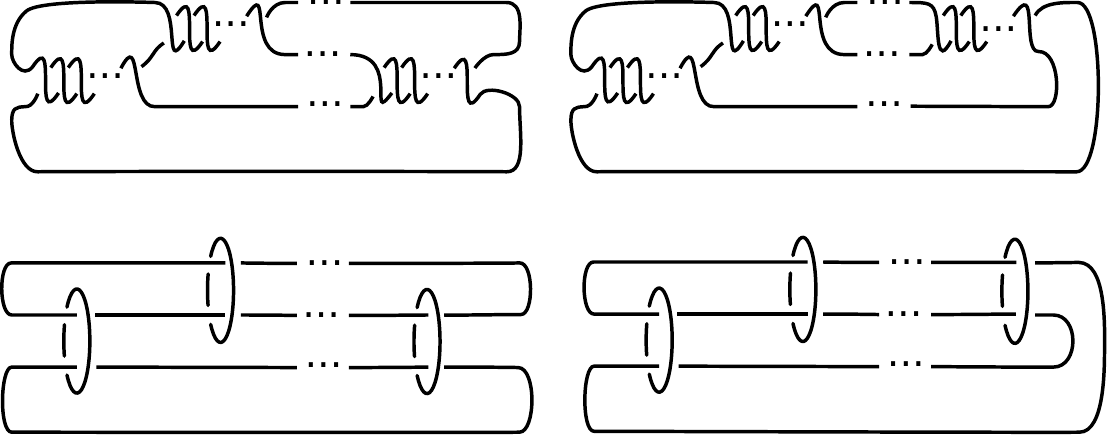}
\caption{The two forms of a 2-bridge link above, and the forms of the fully augmented 2-bridge link (without half-twists) below.}
\label{Fig:2BridgeForms}  
\end{figure}

To obtain the polyhedra, cut the fully augmented 2-bridge link along the geodesic surface of the projection plane. This cuts each of the $2$-punctured disks bounded by a crossing circle in half. The 3-sphere is cut into two pieces, one above and one below the projection plane. For each half, we cut open half discs and flatten them in the projection plane. Lastly, shrink the link components to ideal vertices. See \reffig{2BridgeCirclePacking}, left. The circle packing giving the polyhedral decomposition of a fully augmented $2$-bridge link is shown in \reffig{2BridgeCirclePacking}, right.

\begin{figure}
  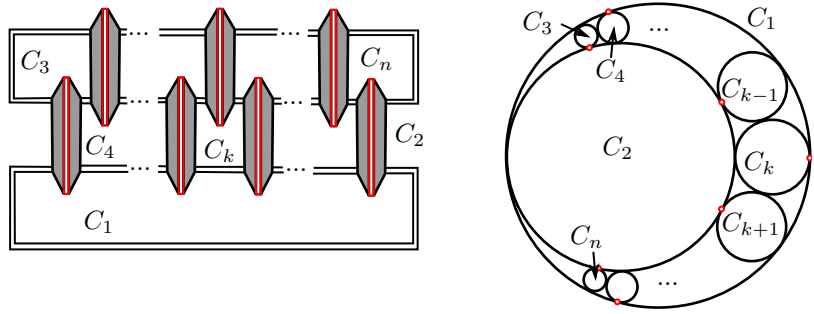
  \caption{How to decompose a fully augmented 2-bridge link into two polyhedra. On the right, the corresponding circle packing.}
  \label{Fig:2BridgeCirclePacking}
\end{figure}

\begin{lemma}\label{Lem:2BridgeCusps}
  The cusp shapes of any fully augmented 2-bridge link complement with no half-twists consist of a $1\times 2$ block of squares, if the cusp is the first or the last in the diagram, or a $2\times 2$ block of squares for all other cusps. 
\end{lemma} 

\begin{proof}
The fully augmented 2-bridge link has a circle packing built of two tangent circles labelled $C_1$ and $C_2$, and a string of circles $C_3, \dots, C_n$, each tangent to both $C_1$ and $C_2$, and $C_j$ tangent to $C_{j-1}$ and $C_{j+1}$ for $j=4, \dots, n-1$, as in \reffig{2BridgeCirclePacking}, right. This circle packing describes each of the two polyhedra making up the link complement.

We need to determine the cusp shapes of the crossing circle cusps. In each polyhedron, the crossing circle cusps correspond to tangencies in the circle packing betweeen $C_2$ and $C_3$, between $C_4$ and $C_1$, and more generally, between $C_{2}$ and $C_{2k+1}$, and between $C_1$ and $C_{2k}$. We take each of these tangent points to infinity to determine the cusp shape. There are two cases.

\textbf{Case 1.} Consider the first and last crossing circles, corresponding to tangencies of $C_2$ and $C_3$, and of either $C_2$ and $C_n$ or $C_1$ and $C_n$ if $n$ is odd or even, respectively. 

Take the point of tangency to infinity. In the case of $C_2$ and $C_3$, there are two circles, $C_1$ and $C_4$, tangent to both $C_2$ and $C_3$, and tangent to each other. Thus the circle packing forms a square, similar to the case of the Borromean rings. When we glue across white faces, we glue an identical square coming from the second polyhedron, and the cusp becomes a $1\times 2$ rectangle. See \reffig{2BridgeCusp}, left. The case of $C_n$ is similar.

\begin{figure}
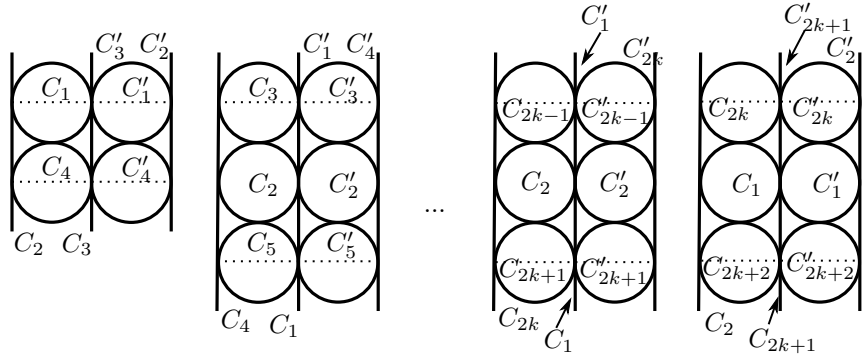
  \caption{Left to right: The first cusp, the next cusp, the $(2k)$-th cusp, and the $(2k+1)$-th cusp.}
  \label{Fig:2BridgeCusp}
\end{figure}

\textbf{Case 2.} For tangencies between circles $C_1$ and $C_{2k}$ or $C_2$ and $C_{2k+1}$, where the circles $C_{2k}$ or $C_{2k+1}$ are not the first or last such circles, when we take the point of tangency to infinity we see a pattern as shown in \reffig{2BridgeCusp}, right. That is, in the first case, circles $C_1$ and $C_{2k}$ become parallel lines. Between them are three circles tangent to both parallel lines, namely $C_{2k-1}$, $C_2$, and $C_{2k+1}$. These circles form the cusp circle packing. Again glue a white face to another white face to obtain the full cusp shape, shown in \reffig{2BridgeCusp}, middle-right. In this case, the cusp shape is a $2\times 2$ rectangle. The case of $C_2$ and $C_{2k+1}$ is similar, and is illustrated in \reffig{2BridgeCusp}, far right. 
\end{proof}

Notice that each crossing circle cusp is tiled by ideal pyramids over a square base. The first and last crossing circle cusps are tiled by two pyramids, the others by four.

The fully augmented 2-bridge link is obtained by gluing all these pyramids together according to the following pattern.

\begin{lemma}\label{Lem:2BridgeGlue}
Consider the fully augmented 2-bridge link, with no half-twists, and cusp shapes as in \reflem{2BridgeCusps}. The gluing is as follows.
\begin{itemize}  
\item The first cusp, which is a $1\times 2$ rectangle, is glued to the top half of the second cusp, with left side gluing by a reflection in the diagonal of positive slope, and right side gluing by reflection in the diagonal of negative slope.

\item The bottom half of the second cusp, another $1\times 2$ rectangle, is glued to the top half of the third cusp, with the left side gluing by reflection in the negative diagonal, and the right side gluing by reflection in the positive diagonal.

\item Inductively, the $1\times 2$ bottom half of the $k$-th cusp glues to the $1\times 2$ top half of the $(k+1)$-th cusp, with left side gluing by reflection in one diagonal, and the right side gluing by reflection in the other diagonal. See \reffig{2BridgeGluing}.
Importantly, diagonals glue to diagonals, and horizontal lines glue to vertical and vice-versa.

\item Finally, on the last $2\times 2$ rectangle, the bottom $1\times 2$ rectangle glues to the final crossing circle cusp, a $1\times 2$ rectangle, again with vertical edges gluing to horizontal and horizontal to vertical, and diagonals gluing to diagonals.
\end{itemize}
\end{lemma}

\begin{figure}
\includegraphics{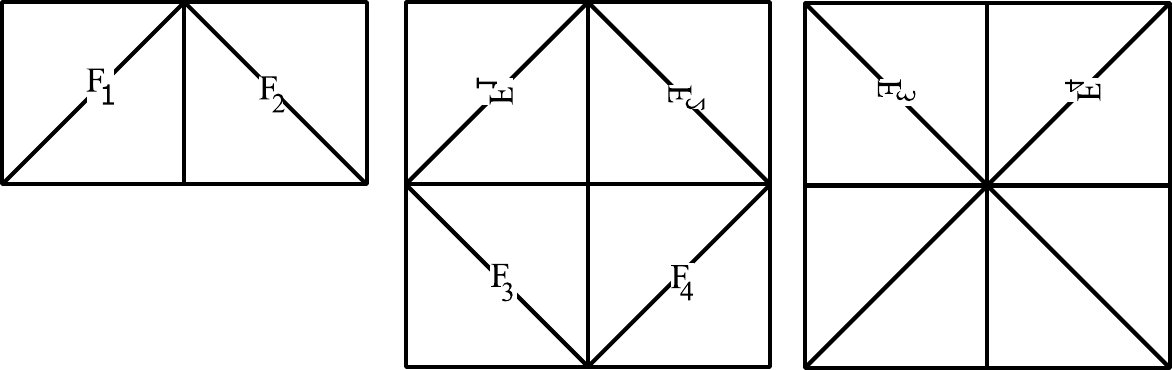}
\caption{How the square faces of pyramids glue to each other.}
  \label{Fig:2BridgeGluing}
\end{figure}

\begin{proof}
As in the case of the link with two crossing circles, this result is obtained by observing the intersections of circles in the circle packings; refer to \reffig{2BridgeCusp}. For the first cusp, the square on the left has ideal vertices $C_1\cap C_2$, $C_2\cap C_4$, $C_3\cap C_4$, $C_3\cap C_1$ in anti-clockwise order. These map to the top square in the middle left, but note this list of vertices now runs in clockwise order, with $C_2\cap C_4$ and $C_3\cap C_1$ in the same locations. It follows that the gluing is a reflection in the diagonal of positive slope. For the square on the right, ideal vertices $C_1'\cap C_3'$, $C_3'\cap C_4'$, $C_2'\cap C_4'$, $C_2'\cap C_1'$ in anti-clockwise order are mapped to the top right square in the middle-left, only now in clockwise order. The gluing is a reflection in the diagonal of negative slope. 

The rest of the squares are treated similarly to obtain the result.
\end{proof}

To Dehn fill, we remove the interiors of the pyramids meeting the cusp, leaving only the square base of each pyramid behind. We will then put in a triangulated solid torus in which the chosen slope bounds a meridian and the boundary is triangulated to match the triangulation on the original manifold.

We know how to fill in a solid torus in the case of the $1\times 2$ rectangles: we just take the double cover of the layered solid torus as in \reflem{DoubleCoverSolidTorus} or the side-by-side solid torus as in \reflem{DoubleLST}.

For the $2\times 2$ rectangle, we see in \reflem{2BridgeSolidTorus} that the double cover of one of the solid tori we have already encountered will always work.

\begin{lemma}\label{Lem:2BridgeSolidTorus}
Suppose $a/b$ is a slope, and $a/b \notin\{0/1, 1/0, \pm 1/1\}$. 
\begin{enumerate}
\item If $a$ is odd and $b$ is even, then let $Y$ be the double cover of the vertical side-by-side solid torus $X$ from \reflem{VerticalSideBySide}, constructed so that $X$ has meridian $a/2b$.
\item If $a$ is even and $b$ is odd, or if $a$ is odd and $b$ is odd, then let $Y$ be the double cover of the horizontal side-by-side solid torus $X$ from \reflem{DoubleLST}, constructed so that $X$ has meridian $2a/b$. 
\end{enumerate}
In either case, $Y$ is a triangulated solid torus whose boundary consists of eight ideal triangles in four symmetric pairs, forming a $2\times 2$ square. The slope $a/b$ is a meridian of $Y$.
\end{lemma}

\begin{proof}

\noindent\underline{Case 1.} Suppose $a/b$ is of the form odd over even.

Let $X$ be the triangulated solid torus that is a ``vertical'' side-by-side of a layered solid torus, constructed so that the meridian of $X$ has slope $a/2b:=m_X$, as in \reflem{VerticalSideBySide}. Note that to build such a triangulation, we walk to the slope $a/b$ in the Farey graph, then identify edges of slope $a/b$ in the final layered tetrahedra, and insert one more tetrahedron. In $\RR^2$, the slope $1/0:=\mu_X$ in $X$ lifts to run from $(0,0)$ to $(0,2)$, and $0/1:=\lambda_X$ lifts to run from $(0,0)$ to $(1,0)$. The meridian of $X$, the slope $a/2b=m_X$, lifts to run from $(0,0)$ to $(2b,2a)$. 

Let $Y$ denote the (horizontal) double cover of $X$. We will show $Y$ gives the required Dehn filling. To do so, we need to show that $a/b$, now written as a slope on $Y$, in the basis for $Y$, bounds a disc. 

In the solid torus $X$, the slope $\mu_X=1/0$ lifts to the curve running from $(0,0)$ to $(0,2)$ in $\RR^2$. The meridian slope $m_X=a/2b$ meets the slope $\mu_X$ a total of
$|a\cdot 0 - 2b\cdot 1| = |2b|$ times, which is even. Thus the curve $\mu_X$ is homotopic to an even power of the core of $X$. Therefore it lifts to a generator of the fundamental group of $\partial Y$. Denote this generator by $\mu_Y = 1/0$. When we lift $\bdy Y$ to $\RR^2$, the lift of $\mu_Y$ still runs from $(0,0)$ to $(0,2)$.

The meridian slope $m_X=a/2b$ of $X$ meets the slope $\lambda_X=0/1$ a total of $|2b\cdot 0 - a\cdot 1| = |a|$ times, which is odd. Thus a second generator of the fundamental group of $\partial Y$ is given by taking two lifts of the curve $\lambda_X=0/1$, lined up end-to-end. Denote this generator by $\lambda_Y$. When we lift $\bdy Y$ to $\RR^2$, $\lambda_Y$ lifts to run from $(0,0)$ to $(0,2)$ in $\mathbb{R}^2$, i.e.\ twice the lift of the corresponding generator in $\bdy X$.

The meridian $m_X=a/2b$ of $X$ lifts to bound a disc in $Y$. Note that the lift runs $a$ times along $\mu_Y$ and $b$ times along $\lambda_Y$. This means the meridian of $Y$ is the slope $a/b$, as desired. 

\vspace{.1in}

\noindent\underline{Case 2.} Suppose $a/b$ is of the form even over odd or is of the form odd over odd.

Let $X$ be the triangulated solid torus that is a ``horizontal'' side-by-side solid torus, constructed so that the meridian of $X$ has slope $2a/b:=m_X$. Note that to build such a triangulation, we walk to the slope $a/b$ in the Farey graph, as in \reflem{DoubleLST}. The solid torus $X$ has generators of the fundamental group given by slopes $\mu_X=1/0$, lifting to run from $(0,0)$ to $(0,1)$ in the $\RR^2$ cover of $\bdy X$, and $\lambda_X = 0/1$, lifting to run from $(0,0)$ to $(2,0)$. The meridian $2a/b$ lifts to run from $(0,0)$ to $(2b,2a)$. 

Let $Y$ denote the (vertical) double cover of $X$. We will show $Y$ gives the required Dehn filling. To do so, we need to show that $a/b$, written in the basis for $Y$, bounds a disc. 

In $X$, the slope $\mu_X = 1/0$ meets the meridian slope $m_X=2a/b$ a total of $|1 \cdot b - 0 \cdot 2a| = |b|$ times, which is odd. Thus $\mu_X$ in $X$ lifts to an arc in $\partial Y$. A generator of the fundamental group of $\partial Y$ is given by taking two lifts of this curve, lined up end-to-end. Denote the resulting closed curve in $Y$ by $\mu_Y$. Its lift runs from $(0,0)$ to $(0,2)$ in $\mathbb{R}^2$.

In $X$, the slope $\lambda_X=0/1$ meets the meridian slope $m_X=2a/b$ of $X$ a total of $|0 \cdot b - 1 \cdot 2a| = |2a|$ times, which is even. Thus the curve $\lambda_X$ is homotopic to an even power of the core of $X$, and it lifts to a closed curve in $Y$. A second generator of the fundamental group of $\partial Y$ is given by taking this lift of $\lambda_X$. Denote it by $\lambda_Y$. When we lift $\bdy Y$ to $\RR^2$, the lift of $\lambda_Y$ is the same as the lift of $\lambda_X$: it runs from $(0,0)$ to $(0,2)$ in $\mathbb{R}^2$. 

The meridian $m_X=2a/b$ of $X$ lifts to bound a disc in $Y$. Note that the lift runs $a$ times along $\mu_Y$ and $b$ times along $\lambda_Y$, hence has the slope $a/b$, as desired.
\end{proof}

\begin{lemma}\label{Lem:AngleStructQuadLST}
Let $a/b \in \QQ\cup\{1/0\}$ be such that $a/b\notin\{0/1, 1/0, \pm 1/1\}$. Let $Y$ denote the triangulated solid torus with meridian $a/b$ of \reflem{2BridgeSolidTorus}. Let $\{\theta_p,\theta_q,\theta_r\}$ be exterior dihedral angles along the boundary of the four-punctured torus forming the $2\times 2$ square, each repeated four times symmetrically, satisfying:
\[ 0<\theta_r<\pi, \quad
-\pi < \theta_p,\theta_q < \pi, \quad
\theta_p+\theta_q+\theta_r = \pi.
\]
Then there exists an angle structure on the triangulated solid torus of \reflem{2BridgeSolidTorus} with these exterior angles. 
\end{lemma}

\begin{proof}
This is automatic from Lemma~\ref{Lem:AngleStructDoubleLST} or~\ref{Lem:VerticalSideBySide}: our solid torus is the double cover of a vertical or horizontal side-by-side solid torus $X$, with meridian of $X$ not one of the slopes $\{0/1,1/0,\pm 1/1, \pm 1/2, \pm 2/1\}$. For such solid tori, the angle structure exists, so it exists for the double cover by lifting angles. 
\end{proof}

\begin{lemma}\label{Lem:VolMaxQuadLST}
Let $T$ be the triangulated solid torus of \reflem{2BridgeSolidTorus}.
If the volume functional has its maximum in the boundary of the space of angle structures, then all tetrahedra of $T$ must be flat. Hence, the volume functional takes its maximum in the interior. 

Moreover, all tetrahedra are flat if and only if exterior angles $(\theta_p, \theta_q, \theta_r)$ are one of $(\pi,0,0)$, $(0,\pi,0)$, $(-\pi,\pi,\pi)$ or $(\pi,-\pi,\pi)$.
\end{lemma}

\begin{proof}
This follows immediately from the similar fact for side-by-side solid tori, Lemmas~\ref{Lem:VolMaxDoubleLST} and~\ref{Lem:ExtAngDoubleLST}, or in the vertical side-by-side case by \reflem{VerticalSideBySide}.
\end{proof}

After removing pyramids from the fully augmented 2-bridge link, and putting in triangulated solid tori satisfying the above lemmas, we have a triangulation of a Dehn filling. To obtain an angle structure, we need gluing equations to be satisfied. Since we already know gluing equations inside the solid tori, we only need to ensure gluing equations hold on the boundaries of these solid tori, where they glue to each other. 

\begin{lemma}\label{Lem:GluingEquations}
Let $L$ be a fully-augmented 2-bridge link with $n>2$ crossing circles (and no half-twists). Let $s_1, \dots, s_n$ be slopes that are all positive or all negative, and further
\[ s_1, s_n \notin \{0/1, 1/0, \pm 1/1, \pm 2/1\}, \mbox{ and }
s_2, \dots, s_{n-1} \notin \{0/1, 1/0, \pm 1/1\}.\]
Label horizontal edges in all crossing circle cusps by $\theta_j$, diagonals by $\alpha_j$, and verticals by $\pi-\theta_j-\alpha_j$. (These are exterior angles). 
Let $T$ be the triangulation of the Dehn filling of $S^3-L$ along these slopes obtained by inserting solid tori as above.
Then if there is an angle structure, exterior angles on the solid tori must satisfy:
\begin{equation}\label{Eqn:Diagonal2B}
  \mbox{Diagonal equations: } \alpha_i = -\alpha_{i+1} \quad \mbox{for all $i$}
\end{equation}
\begin{equation}\label{Eqn:Interior2B}
  \mbox{Interior equations: }  2(\theta_i+\alpha_i) -\theta_{i-1}  -\theta_{i+1} = 0 \quad \mbox{for $2\leq i \leq n-1$}
\end{equation}
\begin{equation}\label{Eqn:End2B}
  \mbox{End equations: } \theta_1+\alpha_1 - \theta_2=0 , \mbox{ and }   \theta_n+\alpha_n -\theta_{n-1} = 0.
\end{equation}
\end{lemma}

\begin{proof}
This follows from the gluing description given above.

Diagonal edges map to diagonal edges, and these are the only edges in this edge class. Thus for all $i$, $(\pi-\alpha_i) +(\pi-\alpha_{i+1}) = 2\pi$, which implies $\alpha_i=-\alpha_{i+1}$, giving \eqref{Eqn:Diagonal2B}.

For the first end equations, the vertical edges with angles $\pi-\theta_1-\alpha_1$ in the first $1\times 2$ cusp glue to the horizontal edges on the left side of the first $2\times 2$ cusp, labelled $\theta_2$. Note that both vertical edges and both horizontal edges are glued to the same edge class. Thus $2(\pi-(\pi-\theta_1-\alpha_1)) + 2(\pi-\theta_2) = 2\pi$. This gives the first end equations in \eqref{Eqn:End2B}.

For the interior equations, the horizontal edge with angle $\theta_{i-1}$ in the $(i-1)$-st cusp glues to both vertical edges in the $i$-th cusp, with angles $\pi-\theta_i-\alpha_i$. In turn, both vertical edges in the $i$-th cusp glue to the horizontal edge with angle $\pi - \theta_{i+1}$ in the $(i+1)$-st cusp. Note that this is true for $2 \leq i \leq n-1$. Thus we require that $\pi - \theta_{i-1} + 2(\pi -(\pi -\theta_i-\alpha_i)) + \pi - \theta_{i+1} = 2\pi$ for $2 \leq i \leq n-1$ . This gives the interior equations in \eqref{Eqn:Interior2B}. 

For the last end equations, both horizontal edges with angle $\theta_{n-1}$ in the last $2 \times 2$ cusp glue to both vertical edges with angle $\pi-\theta_n-\alpha_n$ in the last $1 \times 2$ cusp. Thus $2(\pi-\theta_{n-1})+2(\pi-(\pi-\theta_n-\alpha_n)) = 2\pi$, giving the last end equation in \eqref{Eqn:End2B}.
\end{proof}

\begin{lemma}\label{Lem:AngleNonempty2B}
For the triangulation on the Dehn filling of the fully augmented 2-bridge link given above, the space of angle structures is non-empty. 
\end{lemma}

\begin{proof}
Because the signs of all the slopes agree, say all are positive, the solid tori are constructed by starting in the Farey triangulation in the triangle $(0,1,1/0)$, and moving either across the edge from $0$ to $1$ or from $1$ to $1/0$. In either case, the slope $1/1$ cannot correspond to the slope covered first, so $\alpha_j$ will never correspond to the slope $\theta_{r_j}$; it will be $\theta_{p_j}$ or $\theta_{q_j}$. Then set all $\alpha_j=0$. This is in the required range of \reflem{AngleStructQuadLST}. 

Because $\alpha_j=0$ for all $j$, the end equations imply $\theta_1=\theta_2$, and the interior equations imply $2\theta_2=\theta_1+\theta_3$, hence $\theta_3=\theta_1$. Inductively assume $\theta_j=\theta_1$ for $j\leq k$ and $k\leq n-1$. Then $2\theta_k=\theta_{k-1}+\theta_{k+1}$, hence $\theta_{k+1}=\theta_1$ as well. Finally, the end equations imply $\theta_{n}=\theta_{n-1}=\theta_1$. So all angles $\theta_j=\theta_1$.

Now by \reflem{AngleStructQuadLST} in the case of the $2\times 2$ square, or by \reflem{AngleStructDoubleLST} in the case of the $1\times 2$ square, an angle structure exists on the solid tori. 
By choice of angles, these satisfy the gluing equations required in \reflem{GluingEquations}. So this gives an angle structure. 
\end{proof}

\begin{lemma}\label{Lem:VolMax2B}
  Volume is maximised in the interior of space of angle structures.
\end{lemma}

\begin{proof}
Suppose volume is not maximised in the interior. Then there is a flat tetrahedron, say in the $i$-th solid torus a tetrahedron is flat. By \reflem{VolMaxQuadLST} the $i$-th solid torus must be a flat solid torus.

A solid torus is flat if and only if the exterior angles $\alpha_i$, $\theta_i$, $\pi-\alpha_i-\theta_i$ are $(0,0,\pi)$ or $(\pi,-\pi,\pi)$, up to permutation, by \reflem{VolMaxQuadLST}.

Cases:
\begin{enumerate}
\item[(1)] $\alpha_i=\theta_i=0$
\item[(2)] $\alpha_i=0, \theta_i=\pi$
\item[(3)] $\alpha_i=\pi, \theta_i=0$
\item[(4)] $\alpha_i=\pi, \theta_i=-\pi$
\item[(5)] $\alpha_i=-\pi, \theta_i=\pi$.
\end{enumerate}

Case (1):
Suppose that $\alpha_i=\theta_i=0$. Then $\alpha_j=0$ for all $j$ by the diagonal equations \eqref{Eqn:Diagonal2B}. As in the proof of \reflem{AngleNonempty2B}, this implies that $\theta_j=\theta_1$ for all $j$. In particular, $\theta_1=\theta_i=0$, so all $\theta_j=0$, so all the solid tori are flat by \reflem{VolMaxQuadLST}. 

Case (2): 
Suppose that $\alpha_i = 0$ and $\theta_i = \pi$. Then $2\theta_i=\theta_{i-1}+\theta_{i+1}$ by the interior equations  \eqref{Eqn:Interior2B}. Since $\theta_i = \pi$, we have $2\pi=\theta_{i-1}+\theta_{i+1}$, which implies that both $\theta_{i-1}$ and $\theta_{i+1}$ are $\pi$. Hence all solid tori are flat by \reflem{VolMaxQuadLST}.

Case (3):
Suppose that  $\alpha_i=\pi$ and $\theta_i=0$.  By the interior equations \eqref{Eqn:Interior2B}, we have $\theta_{i+1}+\theta_{i-1}-2\theta_i = 2\pi$. Now $\theta_i=0$ gives $\theta_{i-1}+\theta_{i+1}=2\pi$, which implies that $\theta_{i-1}$ and $\theta_{i+1}$ are $\pi$. Then diagonal equations plus these results imply $\alpha_{i-1}=-\pi$ and $\theta_{i-1}=\pi$. This is case (5). We show below that all tetrahedra are flat. 

Case (4):
Suppose that $\alpha_i = \pi$ and $\theta_i = -\pi$.

First suppose $i$ is even, where $1 < i \leq n$. Then $\alpha_j = \pi$ for $j=2k$ and $\alpha_j = -\pi$ for $j = 2k+1$. In particular, $\alpha_1=-\pi$.
By the end equations \eqref{Eqn:End2B}, we have $\theta_2 = \theta_1 - \pi$.
By the interior equations \eqref{Eqn:Interior2B}, we have
$2\theta_j-\theta_{j-1}-\theta_{j+1} = 2\pi$ for $j=2k+1$, and
$2\theta_j-\theta_{j-1}-\theta_{j+1} = - 2\pi$ for $j=2k$. In particular, when $j=2k=2$, this implies $\theta_1=\theta_3$. 

Now inductively assume that $\theta_{2j+1} = \theta_1$ for $j \leq k$ and $\theta_{2j}  = \theta_1 - \pi$ for $j \leq k$, and $2k+3 \leq n-1$. 
Then the interior equations imply:
\[
2\theta_{2k+1}-\theta_{2k}-\theta_{2k+2} =
2\theta_1 - \theta_1 + \pi - \theta_{2k+2} = 2\pi,
\]
and so $\theta_{2k+2} = \theta_1 - \pi$.
Moreover,
\[
2\theta_{2k+2} - \theta_{2k+1} - \theta_{2k+3}
= 2(\theta_1 - \pi) - \theta_1 - \theta_{2k+3} = -2\pi,
\]
thus
$\theta_{2k+3} = \theta_1$.
Finally, the end equation implies:
\[
\theta_n = \begin{cases}
  \theta_{n-1} - \pi = \theta_1-\pi \mbox{ if $n$ even} \\
  \theta_{n-1} + \pi = (\theta_1-\pi) +\pi = \theta_1 \mbox{ if $n$ odd}.
\end{cases}
\] 

Since our fixed $i$ is even, we have $\theta_1 = \theta_i +\pi = -\pi + \pi = 0$. This implies that $\theta_j = 0$ for all $j$ even and $\theta_j = -\pi$ for all $j$ odd. It follows that all tetrahedra are flat by \reflem{VolMaxQuadLST}.

Now suppose $i$ is odd, $\alpha_i=\pi$, $\theta_i=-\pi$.
Then $\alpha_j=-\pi$ for $j$ even and $\alpha_j=\pi$ for $j$ odd. In particular, $\alpha_1=\pi$. As above, the end and interior equations imply that $\theta_j=\theta_1$ when $j$ is odd, and $\theta_j=\theta_1+\pi$ when $j$ is even. So again, $-\pi = \theta_i=\theta_1 = \theta_j$ for all $j$ odd, and $0=\theta_1+\pi= \theta_j$ for all $j$ even. Thus again all tetrahedra are flat. 

Case (5):
Now suppose that $\alpha_i = -\pi$ and $\theta_i = \pi$. This case is similar to Case (4) above.
When $i$ is odd, one can show $\theta_{2j+1}=\theta_1$ and $\theta_{2j}=\theta_1-\pi$ for all $j$, implying $\theta_1=\theta_i=\pi =\theta_{2j+1}$, and $\theta_{2j}=0$. Thus all tetrahedra are flat. 

When $i$ is even, one can show $\theta_{2j+1}=\theta_1$ and $\theta_{2j}=\theta_1+\pi$ for all $j$, implying $\pi = \theta_i=\theta_1+\pi$, so $\theta_1=\theta_{2j+1}=0$, and $\theta_{2j}=0+\pi$, so again all tetrahedra are flat. 
\end{proof}

\begin{theorem}\label{Thm:Main2Bridge}
Let $L$ be a fully-augmented 2-bridge link with $n>2$ crossing circles (and no half-twists). Let $s_1, s_2, \dots, s_n\in \QQ\cup\{1/0\}$ be slopes, one for each crossing circle, that are all positive or all negative. Suppose finally that $s_1$ and $s_n$ are the slopes on the crossing circles on either end of the diagram, and the slopes satisfy: 
\[ s_1, s_n \notin \{0/1, 1/0, \pm 1/1, \pm 2/1\}, \mbox{ and }
s_2, \dots, s_{n-1} \notin \{0/1, 1/0, \pm 1/1\}.\]
Then the manifold obtained by Dehn filling $S^3-L$ along these slopes on its crossing circles admits a geometric triangulation.
\end{theorem}

\begin{proof}
With these slopes, there exists a triangulated solid torus with meridian $s_j$ and boundary triangulated by a number of triangles matching that on the crossing circle boundary. Topologically, the Dehn filling is given by triangulating the solid tori and gluing them together. By \reflem{AngleNonempty2B}, the result admits an angle structure. By \reflem{VolMax2B}, the volume is maximised in the interior of the angle structure. The Casson--Rivin Theorem, \refthm{CassonRivin}, then implies that the triangulation is geometric. 
\end{proof}

\bibliographystyle{amsplain}
\bibliography{biblio}

\end{document}